\documentclass[11pt, reqno]{amsart}
\usepackage{euscript,amscd,amsgen,amsfonts,amssymb,latexsym}

\newcommand{\eqdef}{\stackrel{\scriptscriptstyle\rm def}{=}}
\usepackage{bbm}
\usepackage{pgfplots, multicol}
\usepackage{mathtools}
\usepackage{paralist}
\usepackage{todonotes}
\usepackage{enumerate}

\newtheorem{theorem}{Theorem}[section]

\newtheorem{proposition}[theorem]{Proposition}
\newtheorem{corollary}[theorem]{Corollary}
\newtheorem{definition}[theorem]{Definition}
\newtheorem{lemma}[theorem]{Lemma}
\newtheorem{remark}[theorem]{Remark}
\newtheorem{observation}[theorem]{Observation}

\newtheorem{example}[theorem]{Example}

\newcommand{\beha}{\begin{enumerate}}
\newcommand{\behe}{\end{enumerate}}
\renewcommand{\epsilon}{\varepsilon}

\newcommand{\R}{{\rm Rot}}
\newcommand{\Per}{{\rm Per}}
\newcommand{\EPer}{{\rm EPer}}

\newcommand{\Or}{\mathcal{O}}

\DeclareMathOperator{\topo}{top}
\DeclareMathOperator{\dist}{dist}

\newcommand{\cM}{\EuScript{M}}

\newcommand{\cH}{\EuScript{H}}
\newcommand{\bR}{{\mathbb R}}
\newcommand{\bC}{{\mathbb C}}

\newcommand{\bN}{{\mathbb N}}
\newcommand{\bQ}{{\mathbb Q}}

\newcommand{\cA}{{\mathcal A}}

\newcommand{\cC}{{\mathcal C}}

\newcommand{\PF}{Perron-Frobenius}

\def\1{1\!\!1}

\def\and{\text{ and }}

        \def\diam{\text{\rm {diam}}}

        \def\conv{\text{{\rm conv}}}

\def\h{{\text h}}

                        \def\^{\widetilde}

\def\Per{{\rm Per}}
\def\inn{{\rm int}}

\def\var{{\rm var}}

\def\Per{{\rm Per}}
\def\EPer{{\rm EPer}}

\def\1{1\!\!1}

\def\rv{{\rm rv}}
\def\inn{{\rm int\ }}

\newtheorem*{thmA}{Theorem \ref{thmcomprotshift}}

\newtheorem*{thmB}{Theorem \ref{thmglobentcont}}

\newtheorem*{thmC}{Theorem \ref{thmmaingenentcomp}}

\newtheorem*{thmD}{Theorem \ref{thmfin}}

\DeclareMathSymbol{\varnothing}{\mathord}{AMSb}{"3F}
\renewcommand{\emptyset}{\varnothing}
\title{On the computability of rotation sets and their entropies}
\thanks{This work was significantly advanced while the authors were supported by the Collaborate@ICERM program at ICERM}

\author{Michael Burr}
\address{Department of Mathematical Sciences, Clemson University, Clemson SC}\email{burr2@clemson.edu}
\thanks{Burr was partially supported by the National Science Foundation Grant CCF-1527193. }

\author{Martin Schmoll}\address{Department of Mathematical Sciences, Clemson University, Clemson SC}\email{schmoll@clemson.edu}
\thanks{Schmoll and Wolf were partially supported by a collaboration grant from the Simons Foundation (\#318898 to Martin Schmoll)}

\author{Christian Wolf}\address{Department of Mathematics, The City College of New York, New York, NY, 10031, USA}\email{cwolf@ccny.cuny.edu}

\begin{document}

\begin{abstract}
Given a continuous dynamical system $f:X\to X$ on a compact metric space $X$ and an $m$-dimensional continuous potential $\Phi:X\to \bR^m$, the (generalized) rotation set $\R(\Phi)$ is defined as the set of all $\mu$-integrals of $\Phi$, where $\mu$ runs over all invariant probability measures. Analogous to the classical topological entropy, one can associate the localized entropy $\cH(w)$ to each $w\in \R(\Phi)$. In this paper, we study the computability of rotation sets and localized entropy functions by deriving conditions that imply their computability.  We then apply our results to study to the case of subshifts of finite type. We  prove that $\R(\Phi)$ is computable and that $\cH(w)$ is computable in the interior of the rotation set. Finally, we construct an explicit example that shows that, in general, $\cH$ is not continuous on the boundary of the rotation set, when considered as a function  of $\Phi$ and $w$. This suggests that, in general, $\cH$ is not computable at the boundary of rotation sets.
\end{abstract}
\keywords{Generalized rotation sets, entropy, thermodynamic formalism, computability}
\subjclass[2010]{Primary 37D35, 37E45, 03D15 Secondary 37B10,  37L40, 03D80}
\maketitle

\section{Introduction}

\subsection{Motivation}

Frequently, the trajectory of a particular orbit in a dynamical system is hard, if not impossible, to determine.  For instance,  computations may be sensitive to the accuracy of the initial conditions.  This difficulty motivates the study of statistical properties of the system.  In this approach, one typically considers averages, or similar statistical computations, of measurements performed at different times.  The mathematical theory supplies several 
objects and invariants, such as the entropy, pressure, and characteristic exponents that give 
insight in the statistical behavior of a system.  In this paper, we study integrals of (vector-valued) potential functions with respect to measures invariant under the dynamics.   In particular, we prove computability results for the set of integrals of these potential functions as well as their localized entropies.

To illustrate the computational challenges, we consider the dynamical system given by the doubling map, i.e., $f:[0,1)\rightarrow[0,1)$ where $f(x)=2x\pmod 1$.  Since computers use binary arithmetic, the standard number types (such as floats or doubles) on a computer represent dyadic rational numbers, i.e., elements of $\mathbb{Z}\left[\frac{1}{2}\right]$.  Since these numbers have a finite binary expansion, a straight-forward calculation shows that each dyadic rational number in $[0,1)$ is eventually mapped to $0$ under iteration. Therefore, computational experiments with dyadic integers might lead to the incorrect hypothesis that $0$ is an attracting fixed point that attracts all $x\in[0,1)$.  Alternately, if one were to symbolically represent rational numbers, such experiments might lead to the incorrect conclusion that every point in $[0,1)$ is preperiodic.  On the other hand, with computability theory, we study the behavior of $x\in[0,1)\setminus\bQ$ even though we may only compute\footnote{In this simple example, it is possible to use symbolic tools to study the behavior of more points, such as the roots of polynomials with integral coefficients.  Since this may not be possible in more sophisticated systems, we do not address such computations here.} the behavior of periodic points.

The main idea behind computability theory is to represent mathematical objects, e.g., points, sets, and functions by convergent sequences produced by a Turing machine (a computer algorithm for our purposes).  We say that a point, set, or function is computable if there exists a Turing machine that outputs an approximation to any prescribed accuracy, for additional details, see Section \ref{sec:compute:basic} and \cite{Tu}.  Using convergent sequences of points instead of single points allows one to study the behavior of a larger class of objects and to increase the precision of an approximation, as needed, to adjust for the sensitivity to the accuracy of the initial conditions.  

In this paper, we provide conditions so that the rotation set, i.e., the set of integrals of potential functions with respect to all invariant measures, and the localized entropy function are computable, i.e., can be approximated to any prescribed accuracy.   Rotation sets appear as natural extensions of Poincar\'e's rotation number for circle homeomorphisms, and, more generally, of pointwise rotation sets for homeomorphisms on the $n$-torus, see \cite{MZ}.  Rotation sets play a role in several areas of ergodic theory and dynamical systems, and
they have been studied recently by several authors, see, e.g., \cite{B, BZ, GL, GKLM, GM, Je, KW1, KW3, MZ, Z}.  These studies include applications to higher-dimensional multifractal analysis, see, e.g., \cite{BQ} and the references therein, ergodic optimization \cite{GL,Je3}, and the study of ground states and zero-temperature measures \cite{KW4}. 

Our results apply directly to subshifts of finite type for which we prove computability of the rotation set $\R(\Phi)$ of a continuous potential $\Phi$ and the localized entropy $\cH(w)$ for all $w\in \inn \R(\Phi)$.  Our results extend, immediately, to systems that can be modeled (via a computable conjugacy) by a symbolic system, such as uniformly hyperbolic systems with a computable Markov partition and certain parabolic systems, see, e.g., \cite{Br1,BG,UW1}.  Other potential applications include systems that can be exhausted by sufficiently large sets on which they are conjugate to symbolic systems, such as certain non-uniformly hyperbolic systems, e.g., \cite{GeW}, systems with shadowing \cite{GMe}, and systems with discontinuous potentials, e.g., the geometric potential in the presence of critical points \cite{PR}. 

In the literature, there are several recent papers that study invariant sets, topological entropy, and other invariants from the computability point of view.  The computability of Julia sets has been particularly popular, see, e.g., \cite{D1, DY, Br1, BBRY, BBY1, BY, BY1, BY2, BY3}. There are several results about the computability of certain specific measures, see \cite{BBRY,GHR} and the references therein, such as a maximal entropy measure or physical measure, the numerical computation of entropy and dimension for hyperbolic systems, see, e.g., \cite{JP1} and \cite{JP2} and the references therein, as well as with the computation of the topological entropy/pressure for one and multi-dimensional shift maps, see, e.g., \cite{HM,Pa, PS1, HS,Sc,SP}. To the best of our knowledge, our attempt is the first to establish computability of an entire entropy spectrum within the space of all invariant measures.  

\subsection{Background material from dynamical systems}\label{Section:Notation}
In this section, we introduce the relevant material from the theory of dynamical systems. Our main objects of study are rotation sets and their associated entropies.  

Let $f:X\to X$ be a continuous map on a compact metric space $X$. Let $\cM$ denote the space of all $f$-invariant Borel probability measures on $X$, endowed with the weak$^\ast$ topology. This makes $\cM$ into a compact, convex, and metrizable topological space. Recall that $\mu\in \cM$ is ergodic if every $f$-invariant set has either measure zero or one. We denote by $\cM_E\subset \cM$ be the subset of ergodic measures. 

We denote the {\em set of all periodic points of $f$ with smallest period $n$} by $\Per_n(f)$. We also call $n$  the {\em prime period} of $x\in \Per_n(f)$. Moreover, $\Per(f)=\bigcup_{n\geq1} \Per_n(f)$ denotes the {\em set of periodic points} of $f$. The elements of $\Per_1(f)$ are the {\em fixed points} of $f$. For $x\in \Per_n(f)$, we denote the unique invariant measure supported on the orbit of $x$ by $\mu_x=1/n(\delta_x+\dots +\delta_{f^{n-1}(x)})$. We also call $\mu_x$ the periodic point measure of $x$. Moreover, we write $\cM_{\rm Per}=\{\mu_x: x\in \Per(f)\}$. We observe that $\cM_{\rm Per}\subset \cM_E$.

Throughout this paper, we assume that $f$ has finite topological entropy (see, e.g., \cite{Wal:81} for the definition of topological entropy).   Given an $m$-dimensional potential  $\Phi=(\Phi_1,\dots,\Phi_m)\in C(X,\bR^m)$, we denote the {\it generalized rotation set} of  $\Phi$ with respect to $f$
by 
 $\R(\Phi)=\R(f,\Phi)$ defined by
\begin{equation} \label{defrotset}
 \R(\Phi)= \left\{\rv(\mu): \mu\in\cM\right\},
\end{equation}
where
\begin{equation}
\rv(\mu)=\left(\int \Phi_1\ d\mu,\dots,\int \Phi_m\ d\mu\right)
\end{equation}
 denotes the {\em rotation vector} of the measure $\mu$. Given $w\in \R(\Phi)$, we call $\cM_\Phi(w)=\{\mu\in \cM: \rv(\mu)=w\}$ the {\em rotation class} of $w$.
 It follows, from the definition, that
the rotation set is a compact and convex subset of $\bR^m$. 

The relevance of rotation sets for understanding the behavior of dynamical systems 
can be seen by considering a 
sequence of potentials $(\Phi_k)_k$ that is dense in $C(X,\bR)$. Let $R_m$ be the rotation
set of the initial $m$-segment of potentials, that is $R_m = \R(\Phi_1,\dots,\Phi_m)$.
It follows, from the representation theorem, that the rotation classes of the
rotation sets $R_m$ form a decreasing sequence of partitions of $\cM$ whose intersections
contain a unique invariant measure. Therefore, for large $m$, the set
$R_m$ provides a fine partition of $\cM$ and acts as a finite dimensional
approximation to the set of all invariant probability measures.
 We say $(R_m)_m$ is a {\em filtration} of $\cM$.
 
 We say  $(\Phi_{\epsilon_n})_n$, where $\Phi_{\epsilon_n}\in C(X,\bR^m)$ is an {\em approximating sequence} of $\Phi$ if $\epsilon_n\to 0$ as $ n\to \infty$ and $\|\Phi_{\epsilon_n}-\Phi\|_\infty<\epsilon_n$ for all $n\in \bN$. Here, $\|.\|_{\infty}$ denotes the supremum norm on $C(X,\bR^m)$.
  
Next, we define the localized entropy of rotation vectors.
Following \cite{Je, KW1}, we define the {\em localized entropy} of $w\in \R(\Phi)$ by
\begin{equation}\label{defH}
\cH(w)=\cH_\Phi(w)\eqdef\sup\{h_\mu(f): \mu \in \cM_\Phi(w)\}.
\end{equation}
Here, $h_\mu(f)$ denotes the measure-theoretic entropy of $f$ with respect to $\mu$ (see \cite{Wal:81} for details).  
We consider systems for which  $\mu\mapsto h_\mu(f)$ is upper semi-continuous
on $\cM$; thus, there exists at least one  $\mu\in \cM_\Phi(w)$
with
\begin{equation}\label{deflocmaxent}
h_\mu(f)=\cH(w).
\end{equation}
In this case, we say that $\mu$ is a {\it localized measure of maximal entropy} at $w$. Moreover,  the upper semi-continuity of $\mu\mapsto h_\mu(f)$ implies that $w\mapsto \cH(w)$ is continuous on $\R(\Phi)$, see \cite{Je}.

\subsection{Statement of the Results.} We continue to use the notation from Section 1.2. Let $f:X\to X$ be a continuous map on a compact  metric space $X$,  and let  $\Phi:X\to \bR^m$ be a  continuous potential. We  assume that $\mu\mapsto h_\mu(f)$ is upper semi-continuous, which guarantees  that the localized entropy function $w\mapsto \cH(w)$ is continuous.
 It follows, from the definitions, that
$\R(\Phi)$ is a compact and convex subset of $\bR^m$. Conversely, for symbolic systems, see Section 5, every compact convex subset of $\bR^m$ can be ``realized" by an appropriate potential  \cite{KW1}. 
Evidently, in order to shed light on the computability of $\R(\Phi)$, we need to assume that $f$ and  $\Phi$ are computable.
We establish, in Theorem \ref{thmrotcomp}, a general criterium for the computability of $\R(\Phi)$. We then prove that this criterium is satisfied for subshifts of finite type (SFT).

\begin{thmA}
Let $f:X\rightarrow X$ be a transitive subshift of finite type with computable distance $d_{\theta}$.  If $\Phi\in C(X,\mathbb{R}^m)$ is computable, then $\R(\Phi)$ is computable.
\end{thmA}

Here, the computability of a set $S$ means that there exists an algorithm that approximates $S$  in the Hausdorff metric by a finite union of computable closed balls. We refer the reader to Section 2.1 for the precise definition.  As a consequence of the proofs of Theorem \ref{thmrotcomp}  and \ref{thmcomprotshift}, we obtain the computability of the maximal radius $r$ of a  ball centered at $w\in \inn \R(\Phi)$ that is contained  in $\R(\Phi)$, see Proposition \ref{prop:radiuscomputable}.

We expect Theorem \ref{thmcomprotshift} to have several applications: For example, Theorem \ref{thmcomprotshift} is applicable for computing  maximizing integrals of one-dimensional potentials that are of interest in the area ergodic optimization, see \cite{Je3} for an introduction to the subject. Furthermore, our result can be applied to obtain computability results for certain optimizing functions that were studied in the context of relative optimization by Garibaldi and Lopes  \cite{GL}.  Theorem \ref{thmcomprotshift} also applies to  the computation of classical rotation sets for certain toral homeomorphisms homotopic to the identity. We refer the interested reader to \cite{Z} to make the connection between these rotation sets and a symbolic system. Finally, Theorem \ref{thmcomprotshift} can be applied to the computation of barycenter sets, see, e.g., \cite{B,HO,Je1,Je2}.

Next, we discuss the computability of  the localized entropy.  One of the difficulties when attempting to compute $\cH(w)$ is that, at any given time, a Turing machine has only access to a finite amount of data associated with an approximation $\Phi_\epsilon$ rather than the precise data of the actual potential $\Phi$. To overcome this problem
we consider the minimal and maximal local entropy functions of $\Phi_\epsilon$  in the closed ball centered at $w$ and radius $r$, which we denote by $h^l_{\Phi_\epsilon}(w,r)$ and $h^u_{\Phi_\epsilon}(w,r)$, respectively, see Equations \eqref{defhu} and \eqref{defhl} in Section \ref{sec:computability:localized}. We show in Proposition \ref{thmrightleft} that if $w\in \inn \R(\Phi)$, then
\begin{equation}\label{eqhu}
\lim_{n\to\infty} h^l_{\Phi_{\epsilon_n}}(w, \alpha \epsilon_n)=\cH_\Phi(w_0)=\lim_{n\to\infty} h^u_{\Phi_{\epsilon_n}}(w, \alpha\epsilon_n)
\end{equation}
for all $\alpha\geq 1$. Moreover, it can be arranged so that $h^l_{\Phi_{\epsilon_n}}(w, \alpha \epsilon_n)$ is increasing and $h^u_{\Phi_{\epsilon_n}}(w, \alpha\epsilon_n)$ is decreasing.  As a consequence, we obtain the following result:
\begin{thmB}\label{thmB}
Let $f:X\to X$ be a continuous map on a compact metric space such that $\mu\mapsto h_\mu(f)$ is upper semi-continuous.  Then the global entropy function $(\Phi,w)\mapsto \cH_\Phi(w)$ is continuous on $\bigcup_{\Phi\in C(X,\bR^m)} \{\Phi\}\times {\rm int\,}\R(\Phi)$.
\end{thmB}
 This result indicates that for points in the interior of the rotation set, it may be sufficient for the computation of $\cH_\Phi$ to compute the localized entropy of an approximation $\Phi_\epsilon$. 
We then use this approach and apply methods from the thermodynamic formalism to compute the localized entropy of $\Phi_\epsilon$. In particular, we  consider potentials 
$\Phi_{\epsilon}$ for which the corresponding one-dimensional potential $v\cdot \Phi_{\epsilon}$ has a unique equilibrium state $\mu_{v\cdot \Phi_{\epsilon}}$ for all $v\in \bR^m$. We refer to Section 2.2 for the definitions and details. It is important to notice that we only require  $\Phi$  to be continuous. However, we have some flexibility in the construction of the approximating potentials, and, in particular, can require  H\"older or Lipschitz continuity, for which there exists a well-developed theory of equilibrium states. We prove the following general result:

\begin{thmC}\label{thmC}
Let $f:X\to X$ be a continuous map on a computable compact 
metric space $X$ such that $\mu\mapsto h_\mu(f)$ is upper semi-continuous. Let $\Phi:X\to \bR^m$ be
computable.  Suppose a computable function $r:\inn\R(\Phi)\rightarrow\bR^+$ is given such that
for all $w\in\inn\R(\Phi)$, $B(w,r(w))\subset\R(\Phi)$.
Suppose that there exists an approximating sequence $(\Phi_{\epsilon_n})_n$
of $\Phi$ such that for all $n\in\bN$ and all $v\in \bR^m$, the potential $v\cdot \Phi_{\epsilon_n}$ has a unique equilibrium state $\mu_{v\cdot \Phi_{\epsilon_n}}$. Moreover, assume  that the functions $n\mapsto \epsilon_n$, $(v,n)\mapsto h_{v\cdot \Phi_{\epsilon_n}}(f)$ and $(v,n)\mapsto \rv(\mu_{v\cdot \Phi_{\epsilon_n}})$ are computable. Then $\cH_\Phi$ is computable on $\inn \R(\Phi)$.
\end{thmC}

We note that the condition of the uniqueness of the equilibrium states is known to hold for
several classes of systems and potentials including Axiom A systems, subshifts
of finite type, and expansive homeomorphisms with specification and
H\"older continuous potentials. 
Recently, there has been significant progress
in generalizing uniqueness results for equilibrium states to wider classes of
shift transformations, non-uniformly hyperbolic maps, and  flows. We refer to the survey article \cite{CP} for further references
and details.

It turns out that Theorem \ref{thmmaingenentcomp} is applicable to subshifts of finite type and computable potentials $\Phi$. One advantage when dealing with SFTs  is that we can work with locally constant computable  approximations $\Phi_\epsilon$. For these potentials, we are able to establish the assumptions in Theorem \ref{thmmaingenentcomp}.  We conclude that  the localized entropy $\cH(w)$
is computable in the interior of $\R(\Phi)$, see Theorem \ref{Hintcomp}.
To the best of our knowledge, Theorems \ref{thmmaingenentcomp} and  \ref{Hintcomp} represent the first results that establish computability of the entropy beyond computing the topological entropy or  measure-theoretic entropy of certain specific invariant measures.  Our proof of Theorem \ref{thmmaingenentcomp} relies on Equation \eqref{eqhu} in a crucial way. It turns out that the right-hand-side identity in Equation \eqref{eqhu} remains true for boundary points of the rotation set.  Our proof, however, of the left-hand side identity does not carry over to the boundary. Obviously, this does not imply that the left-hand side identity of Equation \eqref{eqhu} does not hold. However, we are able to prove the following:

\begin{thmD}
Let $f:X\to X$ a one-sided full shift over an alphabet with 4 symbols. Then, there exists a potential $\Phi\in C(X, \bR^2)$  and a sequence
of locally constant potentials $\Phi_{\epsilon_n}:X\to \bR^2$ with $\lim_{n\to\infty} ||\Phi-\Phi_{\epsilon_n}||_\infty =0$ such that the following holds.
\begin{itemize}
\item $\partial \R(\Phi)$ is an infinite polygon with a smooth exposed point $w_\infty$ and
\item 
$0=\lim_{n\to \infty}h^l_{\Phi_{\epsilon_n}}(w_\infty, \epsilon_n) < \cH(w_\infty)=\lim_{n\to \infty}h^u_{\Phi_{\epsilon_n}}(w_\infty, \epsilon_n)=\log 2.$
\end{itemize}
\end{thmD}
 One consequence of Theorem \ref{thmfin} is that  one can  in general not extend the continuity of global entropy function $(\Phi,w)\mapsto \cH_\Phi(w)$, see Theorem \ref{thmglobentcont}, to the  boundary of the rotation set. This  suggests that localized entropy is, in general, not computable at the boundary of rotation sets.

 \subsection{Outline of Paper}
 
This paper is organized as follows: In Section \ref{sec:2}, we review some
basic concepts from  computational analysis and the thermodynamic formalism.
In Section 3, we discuss the computability of rotation sets for computable maps on compact computable metric spaces.
Section 4 is devoted to the study of the localized entropy function for continuous maps on compact metric spaces.
 In Section 5, we apply the results from Section 3 to subshifts of finite type and establish the computability of their
 rotations sets. Section 6 is devoted to the proof of Theorem \ref{Hintcomp}. Finally, in Section 7, we construct an example
 which shows that the global entropy function is, in general, discontinuous at the boundary of the rotation sets.

\
%---------------------------------------------------------------------------
\section{Preliminaries}\label{sec:2}
%-------------------------------------------------------------------------

In this section, we discuss the relevant background material. We continue to use the notation from Section \ref{Section:Notation}.

\subsection{Basics from computability theory}\label{sec:compute:basic}
We are interested in the feasibility of computational experiments on rotation sets and entropies.  Computability theory allows us to guarantee the correctness and accuracy of our computational experiments.  In this section, we recall that a computer can approximate only countably many real numbers, so, without an accuracy guarantee, a computer experiment might miss interesting behaviors away from this collection of approximable numbers.

For a more thorough discussion of these topics see, e.g.,  \cite{BHW, BY,GHR, K,RW}.  We use different, but closely related, definitions to those in \cite{BY} and \cite{GHR} as well as mirroring their notation in order to allow for cross referencing.  Throughout this discussion, we use a bit-based computation model, such as a Turing machine, as opposed to a real RAM model \cite{S} (where these questions are trivial).  One can think of the set of Turing machines as a particular, countable set of functions; we denote $\phi(x)$ as the output of the Turing machine\footnote{We exclusively use the upper-case $\Phi$ for potentials and lower-case $\phi$ for Turing machines.  Both notations are fairly common in the respective literature.} $\phi$ on input $x$.

We begin by defining the spaces that we study as well as the computable points in these spaces.

\begin{definition}[cf {\cite[Definition 2.2]{GHR}}]
A {\em computable metric space} $(X,d_X, \mathcal{S}_X )$ is a separable metric space $(X,d_X)$ together with a dense sequence $\mathcal{S}_X = (s_1,s_2,\dots)$ with $s_i\in X$, i.e., an injective function $\mathcal{S}_X: \bN \rightarrow X$ with dense image.  
\end{definition}
If the metric or dense subset is clear from context, we may drop those from the notation.

\begin{definition}[cf {\cite[Definition 1.2.1]{BY}}]\label{def:computable}
Let $X$ be a computable metric space.  An {\em oracle} for $\alpha\in X$ is a function $\phi$ such that on input $n$, $\phi(n)$ is a natural number so that $d_X(\alpha,s_{\phi(n)})<2^{-n}$.  Moreover, we say $\alpha$ is {\em computable} if there is a Turing machine $\phi$ which is an oracle for $\alpha$.
\end{definition}

To make this definition more explicit, we begin with an example for real spaces.  For computational purposes, we often set $\mathcal{S}_{\mathbb{R}}$ to be the rationals or the dyadic numbers, $\mathbb{Z}\left[\frac{1}{2}\right]$, since both can be represented exactly on a computer.  In our discussions we always use the rational points $\mathbb{Q}^m$ in $\mathbb{R}^m$, such as in the following one-dimensional example:

\begin{example}
Let $X=\mathbb{R}$ and $\mathcal{S}_{\mathbb{R}}=\mathbb{Q}$.  For a real number $\alpha$, an oracle for $\alpha$ is a function $\phi$ such that on input $n$, $\phi(n)$ is a rational number so that $|\alpha-\phi(n)|<2^{-n}$.
\end{example}

Since there are only countably many Turing machines, there are only countably many computable points\footnote{There are weaker notions of computability, e.g., where $s_{\phi(n)}\rightarrow \alpha$ without a guarantee on the speed of convergence.  Many of our theorems can be stated with weaker hypotheses to allow  for this and other types of computability.  We leave the details to the interested reader.} in any $X$.  In the case of real numbers, the computable numbers include the rational and algebraic numbers as well as some transcendental numbers, such as $e$ and $\pi$.  

We extend the notion of computability to functions.  For functions, it makes sense to focus on inputs which can be arbitrarily approximated, e.g., computable inputs.

\begin{definition}[cf {\cite[Definition 1.2.5]{BY}}]\label{def:computablefunction}
Let $X$ and $(Y,(t_1,t_2,\dots))$ be computable metric spaces where $t_i\in Y$ is the fixed dense sequence.  Suppose that $S\subset X$.  A function $f:S\rightarrow Y$ is {\em computable} if there is a Turing machine $\psi$ such that for any oracle $\phi$ for $\alpha\in S$, $d_Y(t_{\psi(\phi,n)},f(\alpha))< 2^{-n}.$
\end{definition}

For example, if $S\subset\mathbb{R}^n$ and $g=(g_1,\dots,g_m):S\rightarrow\mathbb{R}^m$, then $g$ is computable iff each $g_i$ is computable.  Observe that, in this definition, $\alpha$ does not need to be computable, i.e., the oracle $\phi$ does not need to be a Turing machine.  In the case where $\alpha$ is computable, however, then $f(\alpha)$ is computable because $\psi(\phi,n)$ is an oracle Turing machine for $f(\alpha)$.  It is straight-forward to see that the composition of computable functions is computable because the output of one Turing machine can be used as the input approximation for subsequent machines.  

\begin{definition}
Let $X$ and $Y$ be computable metric spaces as above.  In this paper, we often consider {\em computable} sequences of points or functions.  Such sequences can be written as functions where one input is the index.  In particular, a sequence $(\alpha_n)_n$ of points in $X$ is computable if it is computable as a function $\bN\rightarrow X$, i.e., there exists a Turing machine $\psi$ such that $\psi(n,k)$ is an approximation to $\alpha_n$ so that $d_X(\alpha_n,s_{\psi(n,k)})<2^{-k}$.

A sequence of computable functions $(\Phi_n)_n$ is a computable sequence if it is computable as a function $\bN\times X\rightarrow Y$, i.e., if there exists a Turing machine $\psi$ such that for any oracle $\phi$ for $\alpha\in X$, $\psi(\phi,n,k)$ is an approximation to $\Phi_n(\alpha)$ so that $d_Y(\Phi_n(\alpha),t_{\psi(\phi,n,k)})<2^{-k}$.

We call a computable sequence of points or functions which is convergent a {\em computable, convergent sequence}.
\end{definition}

We observe that in the definition above, each $\alpha_n$ is computable because for each $\alpha_n$ has an oracle Turing machine $\psi_n(k)=\psi(n,k)$.  A similar argument shows that $\Phi_n$ in the definition above is a computable function.  We observe that the existence of a computable, convergent sequence converging to $\alpha\in X$ is different from the existence of a convergent sequence of computable points $(\alpha_n)_n$ converging to $\alpha$.  In particular, since the computable points are dense in $X$, every $\alpha\in X$ has such a convergent sequence.  Moreover, for each $\alpha_n\in X$, there is a, possibly distinct, Turing machine $\phi_n$ for each $\alpha_n$.  Unless the $\alpha_n$'s are generated in a uniform way via a single Turing machine, as above, then one would need an infinite amount of information, i.e., Turing machines for all of the $\alpha_n$'s, to be able to work with the sequence.  Note also that $\alpha\in X$ is computable if and only if there is a computable, convergent sequence converging to $\alpha$.  

Since the definition for a computable function uses any oracle for $\alpha$ and applies even when $\alpha$ is not computable, we can conclude that for any sufficiently close approximation $x$ to $\alpha$, $f(x)$ approximates the value of $f(\alpha)$, i.e., $f$ is continuous.

\begin{lemma}[cf {\cite[Theorem 1.5]{BY}}]\label{lem:continuous}
Let $X$ and $Y$ be computable metric spaces, $S\subset X$, and $f:S\rightarrow Y$.  If $f$ is computable, then $f$ is continuous.
\end{lemma}

This continuity of $f$ in the lemma above can be made more precise as follows: Fix $n\in\mathbb{N}$.  Since, in the definition of a computable function, $\psi$ can be applied to any oracle $\phi$ for $\alpha$, the correctness of the output is dependent only on the accuracy to which $\phi$ is computed within the Turing machine $\psi$.  The accuracy to which $\phi$ is computed is finite since the algorithm terminates.  Hence, if $\beta$ is sufficiently close to $\alpha$, then there is an oracle $\phi'$ for $\beta$ which agrees with $\phi$ up to the computed accuracy, and the output of $\psi(\phi,n)$ equals the output for $\psi(\phi',n)$.  By the upper bound on the error of the output of $\psi$, we can bound the error between $f(\alpha)$ and $f(\beta)$.  In particular, we can use the maximum precision to which the oracle $\phi$ is queried for the following result:

\begin{lemma}[cf {\cite[Theorem 1.6]{BY}}]\label{localmodulusofcontinuity}
Let $X$ and $Y$ be computable metric spaces, $S\subset X$, and $f:S\rightarrow Y$.  If $f$ is computable, then there is a computable function $g:S\times\mathbb{N}\rightarrow \mathbb{N}$ such that if $\alpha,\beta\in S$ and $d_X(\alpha,\beta)<2^{-g(\alpha,k)}$, then $d_Y(f(\alpha),f(\beta))<2^{-k}$.  In this case, we say that $f$ has a {\em computable local modulus of continuity.}  Since $g$ is computable, there exists a Turing machine $\mu$ such that for any oracle $\phi$ for $\alpha$, $g(\alpha,k)=\mu(\phi,k)$.
\end{lemma}

In some cases, we can extend the local modulus of continuity of Lemma \ref{localmodulusofcontinuity} to a global modulus of continuity.  In order to do this, we need a notion of computability for subsets of $X$.  We use the {\em Hausdorff distance} to determine the accuracy of an approximation.  The Hausdorff distance between two {\em compact} subsets $A$ and $B$ of a metric space $X$ is
$$
d_H(A,B)=\max\left\{\max_{a\in A}d_X(a,B),\max_{b\in B}d_X(b,A)\right\}.
$$
In words, the Hausdorff distance is the largest distance of a point in one set to the other set.
\begin{definition}
Let $(X,(s_1,s_2,\dots))$ be a computable metric space.  Let $S\subset X$ be compact.  We say $S$ is {\em computable} if there exists a Turing machine $\psi$ such that on input $n$, $\psi(n)$ is a finite collection of pairs $\{(k_i,l_i)\}$, where $k_i$ is a natural number and $l_i$ is an integer, representing closed balls $\overline{B}(s_{k_i},2^{-l_i})$ centered at $s_{k_i}$ and of radius $2^{-l_i}$ such that
$$
d_H\left(\bigcup_i \overline{B}(s_{k_i},2^{-l_i}),S\right)<2^{-n}.
$$
\end{definition}

Suppose that $X$ is a computable vector metric space.  Suppose that $S$ is a compact, convex, and computable set.  In this case, let $C$ be the union of balls produced by the Turing machine $\phi(n)$ from the definition of a computable set.  We observe that the boundary of $S$ lies in a tubular neighborhood of radius $2^{-n}$ of the boundary of $C$.  Any point not in this tubular neighborhood is guaranteed to be either interior to both $S$ and $C$ or exterior to both.

Observe that when $S\subset X$ is compact and computable and $f:S\rightarrow Y$ is a computable function, we can cover $S$ with finitely many balls and compute a local modulus of continuity on each ball.  By comparing these (and taking smaller balls if necessary), we can derive a global modulus of continuity.

\begin{lemma}[cf {\cite[Section 1.2]{BY}}]\label{globalmodulusofcontinuity}
Let $X$ and $Y$ be computable metric spaces, $S\subset X$, and $f:S\rightarrow Y$.  Suppose that $S$ is compact and computable.  If $f$ is computable, then the computable function $g$ in Lemma \ref{localmodulusofcontinuity} can be extended over all of $S$.  In other words, there is a computable function $g:\mathbb{N}\rightarrow \mathbb{N}$ such that if $\alpha,\beta\in S$ and $d_X(\alpha,\beta)<2^{-g(k)}$, then $d_Y(f(\alpha),f(\beta))<2^{-k}$.  In this case, we say that $f$ has a {\em computable global modulus of continuity.}  Since $g$ is computable, there exists a Turing machine $\mu$ such that for any $k$, $g(k)=\mu(k)$.
\end{lemma}

\subsection{The thermodynamic formalism}\label{sec:thermodynamics} 
A detailed discussion of 
the thermodynamic formalism can be found in \cite{Bo, PaPo, Ru, Wal:81}. Here, we briefly recall some of the relevant
facts. Let $f:X\to X$ be a continuous map on a compact metric space $X$.
Given a continuous one-dimensional potential $\Phi:X\to \bR$, we denote the topological pressure of $\Phi$ (with respect to $f$) by $P_{\rm top}(\Phi)$ and the topological entropy of $f$ by $h_{\rm top}(f)$, see~\cite{Wal:81} for the definition and further details. We recall that $h_{\rm top}(f)=P_{\rm top}(0)$.
The topological pressure satisfies the well-known variational principle, namely,

\begin{equation}\label{eq111}
P_{\rm top}(\Phi)=
\sup_{\mu\in \cM} \left(h_\mu(f)+\int_X \Phi\,d\mu\right).
\end{equation}
A measure
$\mu\in \cM$ that attains the supremum in Equation \eqref{eq111} is
called an equilibrium state (or equilibrium measure) of the potential $\Phi$. We denote the set of all equilibrium states of $\Phi$ by $ES(\Phi)$. Note that $ES(\Phi)$ is a compact and convex subset of $\cM$.
Moreover, if the map $\mu\mapsto h_\mu(f)$ 
is upper semi-continuous, then $ES(\Phi)\not=\emptyset$ in which case $ES(\Phi)$ contains at least one ergodic equilibrium state.
Given $\gamma>0$, we say $\Phi$ is {\em H\"older continuous with exponent $\gamma$} if there exists a $C>0$ such that $\|\Phi(x)-\Phi(y)\|\leq Cd(x,y)^\gamma$ for all $x,y\in X$. We denote the 
 space of all H\"older continuous potentials with exponent $\gamma$ by $C^\gamma(X,\bR).$ Analogously, we denote the space of H\"older continuous functions from $X$ to $\bR^m$ by $C^\gamma(X,\bR^m)$. 
 
 In \cite{KW1}, the authors discuss the class of systems with strong thermodynamic properties (STP). Roughly speaking, STP systems are those systems for which the topological pressure has strong regularity properties. The class of STP systems includes subshifts of finite type,  uniformly hyperbolic systems, and expansive homeomorphisms with specification.
We refer the interested reader to \cite{KW1} for details.
 In the following list, we highlight certain properties of the topological pressure that hold for many classes of systems including STP systems.
 \begin{enumerate}
\item $h_{\rm top}(f)<\infty$;
\item  The entropy map $\mu\mapsto h_\mu (f)$ is upper semi-continuous;
\item  The map $\Phi\mapsto  P_{\rm top}(\Phi)$ is real-analytic on $C^\gamma(X,\bR)$;
\item \label{ref4} Each  potential $\Phi \in C^\gamma(X,\bR)$ has a
unique equilibrium state $\mu_\Phi$. Furthermore,
$\mu_\Phi$ is ergodic, and, given $\Psi\in C^\gamma(X,\bR)$, we have
\begin{equation}\label{eqdifpre}
\frac{d}{dt} P_{\rm top}(\Phi + t\Psi )\Big|_{t=0}= \int_X \Psi
\,d\mu_\phi.
\end{equation}

\end{enumerate}

Next, we discuss an application of the thermodynamic formalism to the theory of rotation sets of STP systems.
 Let $f:X\to X$ be an STP system and let $\Phi\in C^\gamma(X,\bR^m)$. Given $v\in \bR^m$, let $\mu_{v\cdot \Phi}$ denote the unique equilibrium state of the potential $v\cdot \Phi=v_1 \Phi_1+\dots+ v_m \Phi_m$.
We have the following result:
\begin{theorem}[\cite{KW1} (see also \cite{GKLM})]\label{thmKW1}
 Let $\Phi\in C^\gamma(X,\bR^m)$ with ${\rm int}\, \R(\Phi)\not=\emptyset$. Then
 \begin{enumerate}
\item[(i)] The map $T_\Phi:\bR^m\to {\rm int}\, \R(\Phi)$, where $v\mapsto \rv(\mu_{v\cdot \Phi})$, is a real-analytic diffeomorphism,
\item[(ii)] For all $v\in \bR^m$, the measure $\mu_{v\cdot \Phi}$ is the unique localized measure of maximal entropy at $T_\Phi(v)$, and
\item[(iii)]  The map $w\mapsto \cH(w)$ is real-analytic on ${\rm int}\, \R(\Phi)$.
\end{enumerate}
 \end{theorem}

\section{Computability of rotation sets. }\label{sec:4}

In this section, we describe a class of computable dynamical systems for which we are able to establish the computability of the rotation set of a computable potential.  Throughout this section, we assume that $f:X\rightarrow X$ is a computable map on a computable compact metric space $X$.  Additionally, we assume that the metric $d_X$ on $X$ is a computable function.  In Section 5, we see that these conditions are satisfied for shift maps. Recall that the output of a Turing machine is a finite set of objects. This fact motivates the approach to consider approximations consisting of finitely many rotation vectors of invariant measures with  finite support. These measures are precisely the periodic point measures $\cM_{\rm Per}$, see Section 1.2 for the definition. On the other hand, there are several classes of dynamical systems whose periodic point measures are dense, see, e.g., \cite{GK} and the references therein. These systems naturally provide good test cases where the approximation by periodic point measures could lead to a proof of the computability of  rotation sets. We have the following result:

\begin{theorem}\label{thmrotcomp}
Suppose that $X$ is a compact computable metric space and that $\Phi\in C(X,\mathbb{R}^m)$ is a computable potential.  Since $X$ is compact, $\Phi$ has a computable global modulus of continuity $\mu$.  Suppose there exists a Turing machine $\psi$ such that $\psi(n)$ consists of (Turing machines which compute):
\begin{enumerate}
\item A finite set of computable points $p_1,\dots,p_k\in X$ and
\item A finite set of pairs $(q_1,m_1),\dots,(q_l,m_l)$ where $q_j$ is a computable periodic point with period $m_j$
\end{enumerate}
such that there exists a partition $\mathcal{P}=\{P_1,\dots,P_k\}$ of $X$ with the following properties:
\begin{enumerate}
\item For each $i$, $p_i\in P_i$,
\item For each $i$, $\diam(P_i)<2^{-\mu(n+1)-1}$, and 
\item For all $\mu\in\mathcal{M}$, there is some $\nu\in\conv(\mu_{q_1},\dots,\mu_{q_l})$ such that $\sum|\mu(P_i)-\nu(P_i)|<2^{-n}$.
\end{enumerate}
Then, $\R(\Phi)$ is computable.
\end{theorem}
Before beginning the proof, we observe that in the statement, it is not necessary to compute the $P_i$'s, it is enough that they merely exist.
\begin{proof}
First, we define $\Phi_n:X\rightarrow\mathbb{R}^m$ as $\Phi_n(x)=\Phi(p_i)$ for $x\in P_i$.  Since the $P_i$'s are not constructed by $\psi$ (they might not even be computable), we cannot construct $\Phi_n$ even though we can work with it, theoretically.  By construction, since $x\in P_i$, $d_X(x,p_i)<2^{-\mu(n+1)}$, the property of the global modulus of continuity, see Lemma \ref{globalmodulusofcontinuity}, implies that $\|\Phi(x)-\Phi(p_i)\|<2^{-n-1}$.  Therefore, for all $\mu\in\cM$, by bringing the norm inside the integral, we see that $\|\rv_\Phi(\mu)-\rv_{\Phi_n}(\mu)\|<2^{-n-1}$.  Hence, we can bound the Hausdorff distance between the corresponding rotation sets $d_H(\R(\Phi),\R(\Phi_n))<2^{-n-1}$.
We note that, in general, $\Phi_n$ is not continuous (i.e., a step function) but we may still use Equation \eqref{defrotset} for the definition of $\R(\Phi_n)$ and refer to it as the rotation set  of $\Phi_n$. 

Second, we find an upper bound on $\Phi$.  We observe that each $p_i$ is computable and $\Phi$ and the Euclidean norm are computable functions, so we can compute $a_i$, which approximates the value of $\|\Phi(p_i)\|$ to an accuracy of $2^{-n-1}$.  Let $K=\max\{a_i\}+2^{-n}$.  Then we show that $K$ is an upper bound on $\|\Phi\|_\infty\geq \|\Phi_n\|_\infty$ as follows:  For all $x\in X$, there is some $p_i$ with $d_X(x,p_i)<2^{-\mu(n+1)}$, which implies that $\|\Phi(x)-\Phi(p_i)\|<2^{-n-1}$.  Therefore, $\|\Phi(x)\|<\|\Phi(p_i)\|+2^{-n-1}\leq a_i+2^{-n}\leq K$.  Moreover, we observe that $K$ is computable.

Next, we prove that the Hausdorff distance between $\R(\Phi_n)$ and the convex hull $C=\conv(\rv_{\Phi_n}(\mu_{q_1}),\dots,\rv_{\Phi_n}(\mu_{q_l}))$ is less than $2^{-n}K$.  Since the convex hull is a subset of $\R(\Phi_n)$, it is enough to prove that for all $\mu\in\cM$, $\dist_{\mathbb{R}^m}(\rv_{\Phi_n}(\mu),C)<2^{-n}K$.  Fix $\mu\in\cM$, by assumption, we know there exists $\nu\in C$ such that $\sum|\mu(P_i)-\nu(P_i)|<2^{-n}$.  Observe that 
\begin{align*}
\rv_{\Phi_n}(\mu)&=\int_X\Phi_nd\mu=\sum \Phi_n(p_i)\mu(P_i)\\
&=\sum\Phi_n(p_i)\nu(P_i)+\sum \Phi_n(p_i)(\mu(P_i)-\nu(P_i))\\
&=\rv_{\Phi_n}(\nu)+\sum \Phi_n(p_i)(\mu(P_i)-\nu(P_i))
\end{align*}
Therefore, using the upper bound on $\Phi_n$ and the assumptions, we conclude

\begin{equation}
\|\rv_{\Phi_n}(\mu)-\rv_{\Phi_n}(\nu)\|\leq \sum\|\Phi_n(p_i)\||\mu(P_i)-\nu(P_i)|\leq 2^{-n}K.
\end{equation}
Finally, we approximate the values of $\rv_{\Phi_n}(\mu_{q_j})$.  For each $j$, let $Q_j=\{q_{j,0},\dots,q_{j,m_j-1}\}$ be the orbit of $q_j=q_{j,0}$.  
Then, $\rv_{\Phi_n}(\mu_{q_j})=\frac{1}{m_j}\sum_{h=0}^{m_j-1}\Phi_n(q_{j,h})$.  Observe that since $f$ is computable, $q_{j,h}$ is also computable for each $h$.  We fix $j$ and $h$ (but allow $i$ to vary).  Since $d_X$ is computable and each $p_i$ and $q_{j,h}$ are computable, we can compute $d_X(p_i,q_{j,h})$ up to an accuracy of $2^{-\mu(n+1)-1}$.  
By construction, there is some $i$ so that $d_X(p_i,q_{j,h})<2^{-\mu(n+1)-1}$.
Our error estimates show that there is at least one $i$ so that we can guarantee that $d_X(p_i,q_{j,h})<2^{-\mu(n+1)}$. Let $i_{j,h}$ be any $i$ that satisfies this inequality.  Therefore, $\|\Phi(p_{i_{j,h}})-\Phi(q_{j,h})\|<2^{-n-1}$.  In this case, we know that 
\begin{equation}
\frac{1}{m_j}\sum_{h=0}^{m_j-1}\Phi(q_{j,h})=\frac{1}{m_j}\sum_{h=0}^{m_j-1}\Phi(p_{i_{j,h}})+\frac{1}{m_j}\sum_{h=0}^{m_j-1}(\Phi(q_{j,h})-\Phi(p_{i_{j,h}}))
\end{equation}
Therefore, 
\begin{equation}\label{eq:computable:rot}
\left\|\rv_{\Phi_n}(\mu_{q_j})-\frac{1}{m_j}\sum_{h=0}^{m_j-1}\Phi(p_{i_{j,h}})\right\|<2^{-n-1}.
\end{equation}
We observe, by the construction above, that the sum in Equation \eqref{eq:computable:rot} is computable.

Then, using any standard convex hull algorithm\footnote{In computational geometry it is common to use the real RAM model of computation, see \cite{marks}.  This model is considered to be unrealistic because it assumes that all real numbers can be represented explicitly.}, see, for example \cite{marks}, we can compute the convex hull 
\begin{equation}
D=\conv\left(\frac{1}{m_j}\sum_{h=0}^{m_j-1}\Phi(p_{i_{j,h}})\right)_{1\leq j\leq l}.
\end{equation}
Since all of the coordinates are computable (rational), a convex hull algorithm can be performed by a Turing machine, so $D$ is computable.  Since the error in the vertices is at most $2^{-n-1}$ and every point within the convex hull is a weighted combination of the vertices, we know that the Hausdorff distance between $C$ and $D$ is at most $2^{-n-1}$.

Combining these steps, we see that the Hausdorff distance between $\R(\Phi)$ and $D$ is at most $2^{-n}(K+1)$.  Therefore, since $K$ is computable, we can compute an approximation to $\R(\Phi)$ with arbitrary precision.  Therefore $\R(\Phi)$ is computable.
\end{proof}

Next, we observe that the construction in Theorem \ref{thmrotcomp} can be applied to conjugate systems.

\begin{corollary}\label{corcomputrot}
Let $X$ and $Y$ be computable metric spaces.  Suppose that $(X,f)$ and $(Y,g)$ are dynamical systems which are conjugate via the homeomorphism $h:X\rightarrow Y$, i.e., $h\circ f=g\circ h$.  Suppose that $h, h^{-1}$ and $\Phi\in C(X,\mathbb{R}^m)$ are computable.  Then
\begin{enumerate}
\item The conjugate potential $\Phi'=\Phi\circ h^{-1}\in C(Y,\mathbb{R}^m)$ is computable,
\item For all $\mu\in\cM_f$, the map $h_\ast:\cM_f\rightarrow\cM_g$ defined by $(h_\ast\mu)(B)=\mu(h^{-1}(B))$ is a bijection, where $\cM_f$ and $\cM_g$ are the $f$ or $g$-invariant probability measures on the corresponding spaces.  Moreover, $(X,\mu,f)$ and $(Y,\h_\ast\mu,g)$ are measure-theoretic isomorphic.
\item We have $\rv_{\Phi}(\mu)=\rv_{\Phi'}(h_\ast\mu)$ and $\R(f,\Phi)=\R(g,\Phi')$.
\end{enumerate}
Moreover, suppose that $X$ satisfies the conditions of Theorem \ref{thmrotcomp}, then $Y$ also satisfies the conditions of Theorem \ref{thmrotcomp}.
\end{corollary}

\begin{remark}
In Section 5, we show that Theorem \ref{thmrotcomp} holds for subshifts of finite type. Thus, Corollary \ref{corcomputrot} establishes the computability of rotation sets for  systems that are computablly conjugate to a subshift of finite type including uniformly hyperbolic and parabolic systems. 
\end{remark}

Next, we establish a criteria for the computability of the maximal radius of a ball that is contained in the rotation set.

\begin{proposition}\label{prop:radiuscomputable}
Suppose that $\R(\Phi)$ is computable with $\inn \R(\Phi)\not=\emptyset$,  and that there exists a Turing machine $\psi$ such that $\psi(n)$ is a convex polytope whose Hausdorff distance to $\R(\Phi)$ is at most $2^{-n}$.  Let $r:\inn\R(\Phi)\rightarrow\mathbb{R}$ be the function such that $r(w)$ is the radius of the largest open ball centered at $w$ contained within $\inn\R(\Phi)$.  Then $r$ is computable.
\end{proposition}
\begin{proof}
Fix $w\in\inn\R(\Phi)$, and suppose that $\phi$ is an oracle for $w$.  The general idea of this proof is to approximate $\R(\Phi)$ and $w$ and use the approximations to approximate $r$.

Using the given oracle, we can construct a polytope $\psi(n+2)=P_{n+2}$ whose Hausdorff distance to $\R(\Phi)$ is at most $2^{-n-2}$.  Observe that since both $P_{n+2}$ and $\R(\Phi)$ are convex, their boundaries lie in tubular neighborhoods of radius $2^{-n-2}$ of each other.

Let $r'=r'_{n+2}:\inn\R(\Phi)\rightarrow\mathbb{R}$ be the function such that $r'(w)$ is the radius of the largest open ball centered at $w$ contained within $P_{n+2}$.  If $w$ is not in $\inn P_{n+2}$, then $r'(w)=0$. Let $v$ be a closest point to $w$ on the boundary of $P_{n+2}$.  Then, by the tubular neighborhood observation, there is some $u$ in the boundary of $\R(\Phi)$ whose distance to $v$ is at most $2^{-n-2}$. Observe, by the reverse triangle inequality, $r'(w)=\|w-v\|\geq \|w-u\|-\|u-v\|$.  Since $u$ is on the boundary of $\R(\Phi)$, we have that $\|w-u\|\geq r(w)$, so $r'(w)\geq r(w)-2^{-n-2}$.  By performing the same argument, but reversing the roles of $P_{n+2}$ and $\R(\Phi)$, we conclude that $|r(w)-r'(w)|\leq 2^{-n-2}$.

Using the oracle $\phi$, let $w'=\phi(n+2)$ be an approximation to $w$ within distance $2^{-n-2}$.  Observe that the ball $B(w,r'(w))$ is contained within $P_{n+2}$.  Then, the ball $B(w',r'(w)-2^{-n-2})$ is completely contained within $B(w,r'(w))$, and, hence, it is contained within $P_{n+2}$.  Therefore, $r'(w')\geq r'(w)-2^{-n-2}$.  By reversing the roles of $w$ and $w'$, we can then conclude that $|r'(w)-r'(w')|\leq 2^{-n-2}$.  Hence, $r'(w')$ and $r(w)$ differ by at most $2^{-n-1}$.

Observe that $r'(w')^2$ is the square of the distance between $w'$, which is a point with rational coordinates, and a defining linear space for the boundary of $P_{n+2}$, which is a hyperplane defined by points with rational coordinates.  Therefore, $r'(w')^2$ can be computed exactly.  Moreover, the square root function is computable, so $r'(w')$ can be computed to an accuracy of at most $2^{-n-1}$ by a Turing machine.  Hence, the difference between $r(w)$ and the approximation to $r'(w')$ is at most $2^{-n}$, and $r$ is computable.
\end{proof}

Observe that, when the conditions of Theorem \ref{thmrotcomp} are satisfied, the conditions of Proposition \ref{prop:radiuscomputable} are satisfied as well.

\section{Computability of localized entropy}\label{sec:computability:localized}
Our goals in this section are twofold: First, we develop a general theory for the localized entropy function of approximations, and, second, we apply this theory to study the computability of the localized entropy at points in the interior of the rotation set.
In Section \ref{sec:7}, we establish 
 that there are fundamental differences between interior  and boundary points of the rotation set.
 Throughout the remainder of this section, we assume that $f:X\to X$ is a continuous map on a compact metric space $X$ with $h_{\rm top}(f)< \infty$, and that the map $\mu\mapsto h_\mu(f)$ is upper semi-continuous.
 Recall that, under these assumptions, for any $\Phi\in C(X,\bR^m)$, the localized entropy function $\cH_\Phi$ is continuous, and, for each $w\in \R(\Phi)$, there exists at least one $\mu\in \cM_\Phi(w)$ with 
$h_\mu(f)=\cH(w)$, i.e., $\mu$ is a localized measure of maximal entropy at $w$.

\subsection{Localized entropies of approximations}
Given $\Phi\in C(X,\bR^m)$, $w_0\in \bR^m$, and $r>0$, we define the maximum and minimum local
entropy  as follows: The maximum local entropy on $\overline{B}(w_0,r)$ is defined by
\begin{equation}\label{defhu}
h^u_\Phi(w_0,r)=\sup\{\cH_{\Phi}(w): w\in \overline{B}(w_0,r)\cap \R(\Phi)\},
\end{equation}
and the minimum local entropy   on $\overline{B}(w_0,r)$ is defined by
\begin{equation}\label{defhl}
h^l_\Phi(w_0,r)=\inf\{\cH_{\Phi}(w): w\in \overline{B}(w_0,r)\cap \R(\Phi)\}.
\end{equation}
Here, we use the conventions $\sup \emptyset =+\infty$ and $\inf \emptyset =-\infty$. Since $w\mapsto \cH_\Phi(w)$ is continuous, it follows that
if $\overline{B}(w_0,r)\cap \R(\Phi)\not=\emptyset$, then the supremum and infimum in the definition of $h^{u/s}_\Phi(w_0,r)$ is actually a maximum or minimum, respectively.

The following result provides a tool to compute the localized entropy of a given potential in terms of the limit of the maximal local entropies of an approximating sequence:
 \begin{proposition}\label{prop101}
Let $\Phi\in C(X, \bR^m)$, and let $(\Phi_{\epsilon_n})_{n}$ be an approximating sequence for $\Phi$. Let $w_0\in \R(\Phi)$, and let $\alpha\geq 1$.
Then 
\begin{enumerate}[(i)]
\item $h^u_{\Phi_{\epsilon_n}}(w_0,\alpha \epsilon_n)\to \cH_\Phi(w_0)$ as $n\to \infty$;
\item
If $\alpha>1$ and $\epsilon_{n+1}<\frac{\alpha-1}{\alpha+1} \epsilon_n$ for all $n\in \bN$, then $(h^u_{\Phi_{\epsilon_n}}(w_0,\alpha \epsilon_n))_n$ is a decreasing sequence.
\end{enumerate}
\end{proposition}\begin{proof}
To prove {\em (i)}, we  observe that since the map $\nu\mapsto h_\nu(f)$ is upper semi-continuous on $\cM$, there exists $\mu\in \cM_\Phi(w_0)$ with $h_\mu(f)=\cH_\Phi(w_0)$. It now follows from $\|\Phi-\Phi_{\epsilon_n} \|_\infty < \epsilon_n$, that $\rv_{\Phi_{\epsilon_n}}(\mu)\in B(w_0,\epsilon_n)$. Since $\alpha\epsilon_n\geq\epsilon_n$, 
\begin{equation}\label{eq44}
\cH_\Phi(w_0)=h_\mu(f)\leq h^u_{\Phi_{\epsilon_n}}(w_0,\alpha \epsilon_n)
\end{equation}
for all $n\in \bN$.
 It follows from the upper semi-continuity of $\nu\mapsto h_\nu(f)$ that we can pick, for each $n\in \bN$, an invariant measure $\mu_n$ with $\rv_{\Phi_{\epsilon_n}}(\mu_n)\in \overline{B}(w_0,\alpha \epsilon_n)$ and
\begin{equation}\label{eq55}
h_{\mu_n}(f)=   h^u_{\Phi_{\epsilon_n}}(w_0,\alpha \epsilon_n).
\end{equation}
We claim  that $\limsup_{n\to\infty} h_{\mu_n}(f)\leq \cH_\Phi(w_0)$. To prove the claim, we consider a subsequence $(n_i)_i $ such that 
\begin{equation}\label{eq444}
\lim_{i\to\infty}  h_{\mu_{n_i}}(f)=\limsup_{n\to\infty} h_{\mu_n}(f)\ \  \mbox{and}\ \ \lim_{i\to \infty} \mu_{n_i} = \nu
\end{equation}
 for some $\nu\in \cM$. The existence  of $\nu$ follows from the compactness of $\cM$.
We obtain
\begin{align*}
\|&\rv_\Phi(\nu)-w_0\|\\
&\leq\|\rv_\Phi(\nu)-\rv_{\Phi}(\mu_{n_i})\|+\|\rv_{\Phi}(\mu_{n_i})-\rv_{\Phi_{\epsilon_{n_i}}}(\mu_{n_i})\|
  + \|\rv_{\Phi_{\epsilon_{n_i}}}(\mu_{n_i})-w_0\| \\
  &<  \|\rv_{\Phi}(\nu)-\rv_{\Phi}(\mu_{n_i})\| + \epsilon_{n_i} + \alpha \epsilon_{n_i}\to 0 \text{ as } \   i\to \infty.
  \end{align*}
We conclude that $\rv_\Phi(\nu)=w_0$, which implies $h_{\nu}(f)\leq \cH_\Phi(w_0)$. On the other hand, the definition of $\nu$ in Equation \eqref{eq444}, in combination with the upper semi-continuity of $\nu\mapsto h_\nu(f)$, implies that $h_\nu(f)\geq\limsup_{n\to\infty} h_{\mu_n}(f)$. We conclude that
\begin{equation}\label{eqclaim}
\limsup_{n\to\infty} h_{\mu_n}(f)\leq \cH_\Phi(w_0)
\end{equation}
 and the claim is proven. Finally, combining Equation \eqref{eqclaim} with Equation \eqref{eq44} and Equation \eqref{eq55} completes the proof of $(i)$.
 
 Next, we prove $(ii)$: Let $\mu_n$ be as in Equation \eqref{eq55}. We claim that $\rv_{\Phi_{\epsilon_n}}(\mu_{n+1})\in \overline{B}(w_0,\alpha \epsilon_n)$.
 We have
 \begin{align*}
\|&\rv_{\Phi_{\epsilon_n}}(\mu_{n+1})-w_0\|\\
&\leq  \|\rv_{\Phi_{\epsilon_n}}(\mu_{n+1})-\rv_\Phi(\mu_{n+1})\|+\|\rv_\Phi(\mu_{n+1})-\rv_{\Phi_{\epsilon_{n+1}}}(\mu_{n+1})\|\\
&\hspace*{3in}  + \|\rv_{\Phi_{\epsilon_{n+1}}}(\mu_{n+1})-w_0\| \\
  &<  \epsilon_{n} +\epsilon_{n+1}+ \alpha \epsilon_{n+1}<\alpha \epsilon_n.
\end{align*}
The final inequality comes from the assumed relationship between $\epsilon_n$ and $\epsilon_{n+1}$ in the theorem statement.  Moreover, this inequality proves the claim.  Finally, statement {\em (ii)} follows from  the definition of $h^u_{\Phi_{\epsilon_n}}(w_0,\alpha \epsilon_n)$ and Equation \eqref{eq55}.
\end{proof}

Next, we consider rotation vectors in the interior of the rotation set. Our  goal is to strengthen Proposition \ref{prop101} for interior points.  We need the following elementary Lemma:
 
 \begin{lemma}\label{lemhs}
 Let $m\in\bN$ and $w_0\in\mathbb{R}^m$.  For all $\epsilon>0$, there exist $w_1,\dots,w_{2m}\in \overline{B}(w_0,2\sqrt{m}\epsilon)$ such that for all $\widetilde{w}_1,\dots,\widetilde{w}_{2m}$ with $\|w_i-\widetilde{w}_i\|<\epsilon$ for $i\in\{1,\dots,2m\}$, we have $B(w_0,\epsilon)\subset\conv(\widetilde{w}_1,\dots,\widetilde{w}_{2m})$.
 \end{lemma}
 
 \begin{proof}
By translation invariance, we may assume, without loss of generality, that $w_0=0$.  Let the $w_i$'s be the points with coordinates $\pm2\epsilon$.  We prove, by induction, that $[-\epsilon,\epsilon]^m\subset\conv(\widetilde{w}_1,\dots,\widetilde{w}_{2m})$.  When $m=1$, there are two points $w_1=2\epsilon$ and $w_2=-2\epsilon$.  By construction, we know that $\widetilde{w}_1>\epsilon$ and $\widetilde{w}_2<-\epsilon$.  Therefore, $\conv(\widetilde{w}_1,\widetilde{w}_2)=[\widetilde{w}_2,\widetilde{w}_1]$, which contains $[-\epsilon,\epsilon]$.

When $m>1$, let $p\in[-\epsilon,\epsilon]^m$.  For each vector $j\in\{\pm 1\}^{m-1}$, we pair the $w_i$'s which agree in the first $m-1$ coordinates.  In particular, let $w_{j,+}$ and $w_{j,-}$ be the $w_i$'s whose first $(m-1)$ coordinates are given by $(w_{j,\ast})_k=2j_k\epsilon$, but whose last coordinate differs, i.e., $(w_{j,+})_m=2\epsilon$, and $(w_{j,-})_m=-2\epsilon$.  By the base case, we know that there is a convex combination $v_j$ of $\widetilde{w}_{j,+}$ and $\widetilde{w}_{j,-}$ such that $(v_j)_m=p_m$.  Let $\pi_m$ be the projection that ignores the last coordinate.  Observe that $\pi_m(v_j)$ is within $\epsilon$ of $(2j_k\epsilon)_{k\in\{1,\dots,m-1\}}$.  Then, by applying the inductive hypothesis to the $\pi_m(v_j)$'s, we get that $p\in\conv(v_j)$.  Since, moreover, each $v_j$ is a convex combination of the $w_i$'s, it follows that $p\in \conv(\widetilde{w}_1,\dots,\widetilde{w}_{2m})$.  Since $p$ is arbitrary, the claim holds.  The desired result holds since $B(w_0,\epsilon)\subset [-\epsilon,\epsilon]^m$.
\end{proof}

 \begin{theorem}\label{thmrightleft}
 Let  $\Phi\in C(X, \bR^m)$, and let $(\Phi_{\epsilon_n})_{n}$ be an approximating sequence  of $\Phi$. Let $w_0\in \inn \R(\Phi)$, $\alpha\geq 1$
 and $r=2\sqrt{m}$.
Then 
\begin{enumerate}[(i)]
\item \label{eqhul} $\lim_{n\to\infty} h^l_{\Phi_{\epsilon_n}}(w_0, \alpha \epsilon_n)=\cH_\Phi(w_0)=\lim_{n\to\infty} h^u_{\Phi_{\epsilon_n}}(w_0, \alpha\epsilon_n)$;
\item If $\alpha>r$ and $\epsilon_{n+1}<\frac{\alpha-r}{\alpha r} \epsilon_{n}$ for all $n\in\bN$ then $(h^l_{\Phi_{\epsilon_n}}(w_0,\alpha \epsilon_n))_n$ is an increasing sequence for $n$ sufficiently large.
\end{enumerate}
\end{theorem}
\begin{proof}
We first prove {\em (i)}.  The second equality in {\em (i)} was shown in Proposition \ref{prop101}. To prove the first equality, let $n\in\mathbb{N}$ so that $\overline{B}(w_0,r\alpha\epsilon_n)\subset \inn \R(\Phi)$.  Let $w_1,\dots,w_{2m}\in\overline{B}(w_0,r\alpha\epsilon_n)$ be the points constructed in Lemma \ref{lemhs}.  Since $\nu\mapsto h_\nu(f)$ is upper semi-continuous, there exist $\mu_1,\dots,\mu_{2m}\in \cM$ with $\rv_\Phi(\mu_i)=w_i$ and $h_{\mu_i}(f)=\cH_\Phi(w_i)$ for all $i\in\{1,\dots,2m\}$.  Let $\widetilde{w}_i\eqdef\rv_{\Phi_{\epsilon_n}}(\mu_i)$ for $i\in\{1,\dots,2m\}$.

First, we observe that $\widetilde{w}_i\in B(w_i,\alpha\epsilon_n)$ for $i\in\{1,\dots,2m\}$ since $\|\Phi-\Phi_{\epsilon_n}\|_\infty<\epsilon_n\leq\alpha\epsilon_n$.  
It follows from the definition of $h_\Phi^l(w_0,r\alpha\epsilon_n)$ and from $h_{\mu_i}(f)\leq \cH_{\Phi_{\epsilon_n}}(\widetilde{w}_i)$ that
\begin{equation}\label{eqin2}
\begin{split}
h_\Phi^l(w_0,r\alpha\epsilon_n)& \leq \min\{h_{\mu_1}(f),\dots,h_{\mu_{2m}}(f)\}\\
&\leq \min\{\cH_{\Phi_{\epsilon_n}}(\widetilde{w}_1),\dots, \cH_{\Phi_{\epsilon_n}}(\widetilde{w}_{2m})\}.
\end{split}
\end{equation}
By Lemma \ref{lemhs}, it follows that $B(w_0,\alpha \epsilon_n)\subset \conv(\widetilde{w}_1,\dots,\widetilde{w}_{2m})$.   The convexity of $\nu\mapsto h_\nu(f)$ implies that 
\begin{equation}\label{eqin3}
\min\{\cH_{\Phi_{\epsilon_n}}(\widetilde{w}_1),\dots, \cH_{\Phi_{\epsilon_n}}(\widetilde{w}_{2m})\}\leq h^l_{\Phi_{\epsilon_n}}(w_0, \alpha\epsilon_n).
\end{equation}
By combining Inequalities \eqref{eqin2} and \eqref{eqin3}, we obtain
$$
h_\Phi^l(w_0,r\alpha\epsilon_n)\leq h^l_{\Phi_{\epsilon_n}}(w_0, \alpha\epsilon_n).
$$
Now, taking the limit as $n\rightarrow\infty$ and using the fact that $\cH_\Phi$ is continuous results in 
\begin{equation}\label{eq33}
\cH_\Phi(w_0)=\lim_{n\to \infty} h_\Phi^l(w_0,r\alpha\epsilon_n)\leq \lim_{n\to\infty} h^l_{\Phi_{\epsilon_n}}(w_0, \alpha\epsilon_n).
\end{equation}
Observe that 
 \begin{equation}
 \lim_{n\to\infty}h^l_{\Phi_{\epsilon_n}}(w_0, \alpha \epsilon_n)\leq \lim_{n\to\infty}h^u_{\Phi_{\epsilon_n}}(w_0, \alpha\epsilon_n).
 \end{equation}
Therefore, {\em (i)}  follows from Inequality \eqref{eq33} and Proposition \ref{prop101}.

The proof of {\em (ii)} is similar to the proof of {\em (i)}, so we omit many of the details.  Suppose that $n$ is large enough so that $\overline{B}(w_0,r(\alpha\epsilon_{n+1}+\epsilon_n))\subset \R(\Phi)$.  We may then choose $w_1,\dots,w_{2m}\in\overline{B}(w_0,r(\alpha\epsilon_{n+1}+\epsilon_n))$ as in Lemma \ref{lemhs}.  By upper semi-continuity, there exist $\mu_1,\dots,\mu_{2m}\in \cM$ such that $\rv_{\Phi_{\epsilon_n}}(\mu_i)=w_i$ and $h_{\mu_i}(f)=\cH_{\Phi_{\epsilon_n}}(w_i)$ for all $i\in\{1,\dots,2m\}$.  Define $\widetilde{w}_i=\rv_{\Phi_{\epsilon_{n+1}}}(\mu_i)$ for $i\in\{1,\dots,2m\}$.  Observe that since $\|\Phi_{\epsilon_n}-\Phi_{\epsilon_{n+1}}\|_\infty\leq\epsilon_n+\epsilon_{n+1}\leq\epsilon_n+\alpha\epsilon_{n+1}$,  $\widetilde{w}_i\in B(w_i,\alpha\epsilon_{n+1}+\epsilon_n)$.

Since $\widetilde{w}_i=\rv_{\Phi_{\epsilon_{n+1}}}(\mu_i)$, it follows that $h_{\mu_i}(f)\leq\cH_{\Phi_{\epsilon_{n+1}}}(\widetilde{w}_i)$.  Since $\rv(\mu_i)\in B(w_0,\alpha\epsilon_{n+1}+\epsilon_n)$, we know that
\begin{equation}\label{eqin3p}
h^l_{\Phi_{\epsilon_n}}(w_0,\alpha\epsilon_{n+1}+\epsilon_n)\leq\min\{\cH_{\Phi_{\epsilon_{n+1}}}(\widetilde{w}_1),\dots,\cH_{\Phi_{\epsilon_{n+1}}}(\widetilde{w}_{2m})\}.
\end{equation}
By Lemma \ref{lemhs}, it follows that $B(w_0,\alpha \epsilon_{n+1}+\epsilon_n)\subset \conv(\widetilde{w}_1,\dots,\widetilde{w}_{2m})$.   The convexity of $\nu\mapsto h_\nu(f)$ implies that 
\begin{equation}\label{eq33p}
\min\{\cH_{\Phi_{\epsilon_{n+1}}}(\widetilde{w}_1),\dots,\cH_{\Phi_{\epsilon_{n+1}}}(\widetilde{w}_{2m})\}\leq h^l_{\Phi_{\epsilon_{n+1}}}(w_0,\alpha\epsilon_{n+1}+\epsilon_n).
\end{equation}
Combining Inequalities \eqref{eqin3p} and \eqref{eq33p}, we have
$$
h^l_{\Phi_{\epsilon_n}}(w_0,\alpha\epsilon_{n+1}+\epsilon_n)\leq h^l_{\Phi_{\epsilon_{n+1}}}(w_0,\alpha\epsilon_{n+1}+\epsilon_n).
$$
Since $\alpha\epsilon_n>\alpha\epsilon_{n+1}+\epsilon_n$ and $\alpha\epsilon_{n+1}\leq\alpha\epsilon_{n+1}+\epsilon_n$, by assumption, the result follows.
\end{proof}
Next, we extend the entropy function $\cH$ by considering  $\Phi$ as a variable. 
\begin{definition}\label{def:parameter}
Let $T\subset C(X,\mathbb{R}^m)\times\mathbb{R}^m$ be the {\em (total) parameter space} of the rotation sets.  In other words, the fibers of the projection onto the first factor are the rotation sets, so for $\Phi\in C(X,\mathbb{R}^m)$, $\pi_1^{-1}(\Phi)=\{\Phi\}\times\R(\Phi)$.  Set theoretically 
$$
T=\bigcup_{\Phi\in C(X,\bR^m)} \{\Phi\}\times \R(\Phi).
$$
\end{definition}

 As a consequence of Theorem \ref{thmrightleft} we obtain the following:

\begin{theorem}\label{thmglobentcont}
Let $f:X\to X$ be a continuous map on a compact metric space such that $\mu\mapsto h_\mu(f)$ is upper semi-continuous.  Then the global entropy function is continuous on $\bigcup_{\Phi\in C(X,\bR^m)} \{\Phi\}\times {\rm int\,}\R(\Phi)$ (cf Definition \ref{def:parameter}).
\end{theorem}

\subsection{Computability at interior points of the rotation set} We now address the question concerning the computability of the localized entropy 
for points in the interior of the rotation set.  Throughout the remainder of this section, we assume that $f:X\rightarrow X$ is a computable map on a computable compact metric space $X$ with a computable metric $d_X$.

The following result provides a computability criteria for interior points.

\begin{theorem}\label{thmHcompgen}
Let $f:X\to X$ be a continuous map on a computable compact 
metric space $X$ such that $\mu\mapsto h_\mu(f)$ is upper semi-continuous. Let $\Phi:X\to \bR^m$ be
computable and let $w_0\in \inn \R(\Phi)$. Suppose that a computable $r>0$ is given such that
$B(w_0,r)\subset \inn \R(\Phi)$. 
Suppose that there exists an approximating sequence $(\Phi_{\epsilon_n})_n$
of $\Phi$ such that $(\epsilon_n)_n$ is computable.
Suppose that there are oracles approximating the functions 
$(n,s)\mapsto h^l_{\Phi_{\epsilon_n}}(w_0, 2^{-s})$ and $(n,s)\mapsto h^u_{\Phi_{\epsilon_n}}(w_0, 2^{-s})$ to arbitrary precision, where $n\in\mathbb{N}$ and $s$ is a real number given by an oracle.  Then $\cH_\Phi(w_0)$ is computable.
\end{theorem}
\begin{proof}
Suppose that $\psi$ is the oracle Turing machine for the computable, convergent sequence $(\epsilon_n)_n$.  Since $r$ is computable, we can take better and better approximations of $r$ until we can guarantee that $r$ is bounded away from zero.  In particular, we can compute positive upper and lower bounds for $r$.  Using this upper bound, we can find an integer $\alpha>1$ so that $\alpha>r$.

Observe that for all computable $\beta>0$ and $n_0\in\mathbb{N}$, we can compute an $n>n_0$ so that $\epsilon_n<\beta$ as follows: We know that such an $n$ exists since $(\epsilon_n)_n$ converges to $0$.  Since $\beta$ is computable, by approximating $\beta$ sufficiently well, we can find a positive lower bound on $\beta$.  Moreover, by computing the $n$-th term of the oracle Turing machine $\psi(n,n)$ for $\epsilon_n$ for each $n>n_0$, one-by-one, we eventually compute an $n$ so that $\epsilon_n$ is less than the lower bound on $\beta$, and, therefore, $\epsilon_n<\beta$.
By passing to a subsequence, we may assume that
\begin{enumerate}
\item $\epsilon_{n+1}<\epsilon_n$,
\item $\epsilon_n<\frac{r}{\alpha}$,
\item $\epsilon_{n+1}<\frac{\alpha-2\sqrt{m}}{2\alpha \sqrt{m}}\epsilon_n$, and
\item $\epsilon_{n+1}<\frac{\alpha-1}{\alpha+1}\epsilon_n$.
\end{enumerate}
Therefore, by Proposition \ref{prop101} and Theorem \ref{thmrightleft}, we know that $(h_{\Phi_{\epsilon_n}}^u(w_0,\alpha\epsilon_n))_n$ is a sequence decreasing to $\cH_{\Phi}(w_0)$ and $(h_{\Phi_{\epsilon_n}}^l(w_0,\alpha\epsilon_n))_n$ is a sequence increasing to $\cH_{\Phi}(w_0)$.  Observe that $\phi(n,\ast)$ is an oracle Turing machine for $\epsilon_n$.  Since $h^l_{\Phi_{\epsilon_n}}(w_0, \alpha \epsilon_n)$ and $h^u_{\Phi_{\epsilon_n}}(w_0, \alpha \epsilon_n)$ can be approximated to any precision and we have an oracle Turing machine for $\epsilon_n$, for any fixed $k$, we can compute approximations $l_n$ and $u_n$ of error less than $2^{-k}$ for $h^l_{\Phi_{\epsilon_n}}(w_0, \alpha \epsilon_n)$ and $h^u_{\Phi_{\epsilon_n}}(w_0, \alpha \epsilon_n)$, respectively.  Since $(h_{\Phi_{\epsilon_n}}^u(w_0,\alpha\epsilon_n))_n$ and $(h_{\Phi_{\epsilon_n}}^l(w_0,\alpha\epsilon_n))_n$ are decreasing and increasing sequences, respectively, it follows that $\cH_{\Phi}(w_0)\in [l_n-2^{-k},u_n+2^{-k}]$.  We can increase $n$, as necessary, so that $u_n-l_n<2^{-k+1}$.  Then, the entire interval has length at most $2^{-k+2}$, so the midpoint of the interval is an approximation to $\cH_{\Phi}(w_0)$ of error at most $2^{-k+1}$.  Since the choice of $k$ is arbitrary, this shows that $\cH_{\Phi}(w_0)$ is computable.
\end{proof}

\begin{remark}\label{remcons} Observe that Theorem \ref{thmHcompgen} implies that if there exists a computable function $r:\inn\R(\Phi)\rightarrow\bR$ such that for all $w\in\inn\R(\Phi)$, $B(w,r(w))\subset\inn\R(\Phi)$, and the functions $(n,s,w)\mapsto h^l_{\Phi_{\epsilon_n}}(w, 2^{-s})$ and $(n,s,w)\mapsto h^u_{\Phi_{\epsilon_n}}(w, 2^{-s})$ are computable, then $\cH_\Phi$ is computable. 
\end{remark}
 We proceed to study these conditions.
We now present a strategy to get a handle on the computability of the local maximal/minimal entropy. The main idea is to  apply the thermodynamic formalism with the goal  to identify  the localized measures of maximal entropy within a family of equilibrium states.  Recall the following from Section \ref{sec:thermodynamics}: Let $\Phi\in C(X,\bR^m)$, and let $w\in \inn \R(\Phi)$. 
Since $\mu\mapsto h_\mu(f)$ is upper semi-continuous, there exists at least one $\mu\in \cM_\Phi(w)$ with  $h_{\mu}(f)=\cH(w)$, that is, $\mu$ is 
a localized measure of maximal entropy at $w$. For $v\in \bR^m$, we consider the one-dimensional  potential $v\cdot \Phi=v_1 \Phi_1+\dots+ v_m \Phi_m$. Recall that $ES(v\cdot \Phi)$ denotes the set of equilibrium states of the one-dimensional potential $v\cdot \Phi$, see Equation \eqref{eq111}. The analogous  upper semi-continuity argument  shows that $ES(v\cdot \Phi)$ is non-empty. It is a result of Jenkinson \cite{Je} that there exists $v\in \bR^m$ and $\mu_{v\cdot\Phi}\in ES(v\cdot \Phi)$ such that $\mu_{v\cdot\Phi}$ is a localized measure of maximal entropy at $w$. Moreover, the variational principle Equation \eqref{eq111} implies that every localized measure of maximal entropy at $w$ belongs to  $ES(v\cdot \Phi)$.
The following result provides an estimate for the norm of $v$:
 \begin{proposition}\label{prop:vbound}
 Let $\Phi\in C(X,\bR^m)$. Let $v\in \bR^m\setminus \{0\}$ and let $\mu_{v\cdot\Phi}\in ES(v\cdot \Phi)$. Let $r=\dist(\rv(\mu_{v\cdot\Phi}),\partial \R(\Phi))$.  Then $||v||\leq \frac{2}{r}h_{\topo}(f)$.
 \end{proposition}
 \begin{proof}
 If $r=0$, then $\rv(\mu_{v\cdot\Phi})\in \partial \R(\Phi)$ and the inequality is trivial. Assume now that $r>0$, in which case $\rv(\mu_{v\cdot\Phi})\in \inn \R(\Phi)$.
 Suppose, for contradiction, that $\|v\|>\frac{2}{r}h_{\topo}(f)$.  Let $H_v(\Phi)$ be the unique supporting hyperplane of $\R(\Phi)$ for which $v$ is the outward pointing normal vector.  By the compactness of $\R(\Phi)$, $F_v(\Phi)=\R(\Phi)\cap H_v(\Phi)$ is a (nonempty) face of $\R(\Phi)$.  
 
 Let $\nu\in\cM$ be an invariant measure with $\rv(\nu)\in F_v(\Phi)$.  Since $h_{\topo}(f)\geq h_{\mu_{v\cdot\Phi}}(f)$ and $h_\nu(f)\geq 0$, we have that $h_\nu(f)\geq h_{\mu_{v\cdot\Phi}}(f) -h_{\topo}(f)$.  Using Equation \eqref{eq111}, we have that 
 \begin{align*}
 P_{\topo}(v\cdot\Phi)&\geq h_\nu(f)+\int v\cdot\Phi d\nu\\
 &\geq h_{\mu_{v\cdot\Phi}}(f)-h_{\topo}(f)+v\cdot \rv(\nu)\\
 &=h_{\mu_{v\cdot\Phi}}(f)-h_{\topo}(f)+v\cdot \rv(\mu_{v\cdot\Phi})+v\cdot(\rv(\nu)-\rv(\mu_{v\cdot\Phi})).
 \end{align*}
 Observe that $v\cdot(\rv(\nu)-\rv(\mu_{v\cdot\Phi}))$ is $\|v\|$ times the distance $\dist(\rv(\mu_{v\cdot\Phi}),H_v(\Phi))$.  Since $H_v(\Phi)$ does not intersect the interior of $\R(\Phi)$, $\dist(\rv(\mu_{v\cdot\Phi}),H_v(\Phi))\geq r$.  Therefore,
 $$
 P_{\topo}(v\cdot\Phi)\geq h_{\mu_{v\cdot\Phi}}(f)-h_{\topo}(f)+v\cdot \rv(\mu_{v\cdot\Phi})+r\|v\|.
 $$
 Using the assumption on $\|v\|$, we find that 
 $$
 P_{\topo}(v\cdot\Phi)> h_{\mu_{v\cdot\Phi}}(f)+v\cdot \rv(\mu_{v\cdot\Phi})+h_{\topo}(f).
 $$
 This implies that 
 $$
 P_{\topo}(v\cdot\Phi)-\left(h_{\mu_{v\cdot\Phi}}(f)+ \int v\cdot\Phi\,d\mu\right)> h_{\topo}(f).
 $$
 Hence, $\mu_{v\cdot\Phi}$ is not an equilibrium state of $v\cdot\Phi$.  This contradiction completes the proof.
\end{proof}

Next, we prove an auxiliary result:

\begin{lemma}\label{lem:covering}
Let $m\in \bN$.  There is a Turing machine $\chi$ such that for oracles $\phi$ and $\psi$ for $r>0$ and $\delta>0$, respectively, produces computable points $p_1,\dots,p_k\in \bQ^m$such that $\overline{B}(0,r)\subset \bigcup_{i=1}^k B(p_i,\delta)$.
\end{lemma}
\begin{proof}
Observe that, by scaling, it is sufficient to prove this for $r=1$.  

Using the oracle for $\delta$, by approximating $\delta$ sufficiently well, we can bound $\delta$ away from zero, and, moreover, we can find an $n$ so that $2^{-n}<\delta$ and let $d=\left\lceil\frac{2^n}{\sqrt{m}}\right\rceil$.  Then, let $p_1,\dots,p_{(2d+1)^m}$ be the points where every coordinate is of the form $\frac{k}{d}$ where $k\in[-d,d]$ is an integer.  For any point $x$ in the unit square, there exists a $p_i$ such that each coordinate is within $d^{-1}$ of the corresponding coordinate of $p_i$.  Therefore, the distance $\|x-p_i\|<\sqrt{m}d^{-1}\leq 2^{-n}<\delta$.
\end{proof}

The following result is the main result of this section:

\begin{theorem}\label{thmmaingenentcomp}
Let $f:X\to X$ be a continuous map on a computable compact 
metric space $X$ such that $\mu\mapsto h_\mu(f)$ is upper semi-continuous. Let $\Phi:X\to \bR^m$ be
computable.  Suppose a computable function $r:\inn\R(\Phi)\rightarrow\bR^+$ is given such that
for all $w\in\inn\R(\Phi)$, $B(w,r(w))\subset\R(\Phi)$.
Suppose that there exists an approximating sequence $(\Phi_{\epsilon_n})_n$
of $\Phi$ such that for all $n\in\bN$ and all $v\in \bR^m$, the potential $v\cdot \Phi_{\epsilon_n}$ has a unique equilibrium state $\mu_{v\cdot \Phi_{\epsilon_n}}$. Moreover, assume  that the functions $n\mapsto \epsilon_n$, $(v,n)\mapsto h_{\mu_{v\cdot \Phi_{\epsilon_n}}}(f)$ and $(v,n)\mapsto \rv(\mu_{v\cdot \Phi_{\epsilon_n}})$ are computable. Then $\cH_\Phi$ is computable on $\inn \R(\Phi)$.
\end{theorem}
\begin{proof}
By Theorem \ref{thmHcompgen} and Remark \ref{remcons}, it is enough to show that $(n,s,w)\mapsto h^l_{\Phi_{\epsilon_n}}(w, 2^{-s})$ and $(n,s,w)\mapsto h^u_{\Phi_{\epsilon_n}}(w, 2^{-s})$ are computable functions.  Fix $w_0\in \inn\R(\Phi)$, and let $\psi$ be an oracle for $w_0$.

The variational principle for the topological entropy, see Equation \eqref{eq111} with $\Phi\equiv 0$, implies that $h_{\topo}(f)=h_{\mu_{0\cdot\Phi_{\epsilon_n}}}(f)$.  By assumption, the map $(v,n)\mapsto h_{\mu_{v\cdot \Phi_{\epsilon_n}}}(f)$ is computable, so we can approximate $h_{\topo}(f)$ to any precision and compute  an upper bound  $h_{\max}$.  Since $r(w_0)$ can be approximated to any precision and positive, we can approximate $r$ to sufficient accuracy so that $r$ is bounded away from zero.  Let $r_{\min}$ be such a lower bound.  

Recalling the argument in the proof of Theorem \ref{thmHcompgen}, for any computable $\beta>0$ and $n_0\in\mathbb{N}$, we can compute an $n>n_0$ so that $\epsilon_n<\beta$.  By passing to a subsequence, we may assume that the $\epsilon_n$'s are all decreasing and $4\sqrt{m}\epsilon_n<r_{\min}$.  Then, let $w_1,\dots,w_n\in \overline{B}(w_0,\frac{1}{2}r_{\min})\subset\inn\R(\Phi)$ as in Lemma \ref{lemhs}.  For each $w_i$ let $\mu_i\in\cM$ be an invariant measure such that $\rv_{\Phi}(\mu_i)=w_i$.  Observe that since $\|\Phi-\Phi_n\|_\infty<\epsilon_n$, $\|\rv_{\Phi}(\mu_i)-\rv_{\Phi_n}(\mu_i)\|<\epsilon_n<\frac{1}{4\sqrt{m}}r_{\min}$.  Therefore, by Lemma \ref{lemhs}, $B(w_0,\frac{1}{4\sqrt{m}}r_{\min})\subset\conv(\rv_{\Phi_n}(\mu_i))_i\subset\R(\Phi_n)$ for all $n$.

By using $\psi$ to compute an approximation $w$ to $w_0$ whose distance to $w_0$ is at most $\frac{1}{8\sqrt{m}}r_{\min}$, we know that $B(w,\frac{1}{8\sqrt{m}}r_{\min})\subset\inn\R(\Phi_{\epsilon_n})$.  Therefore, by Proposition \ref{prop:vbound}, we can compute an upper bound $R$ for $\|v\|$ that applies to all $\Phi_{\epsilon_n}$ and all $w$ within $\frac{1}{8\sqrt{m}}r_{\min}$ of $w_0$.  Throughout the remainder of this proof, we restrict our attention to the closed ball $B=\overline{B}(0,R)$ in $\mathbb{R}^m$.

Since $B$ is a compact set, for any fixed $n$, the computable functions $v\mapsto h_{\mu_{v\cdot \Phi_{\epsilon_n}}}(f)$ and $v\mapsto \rv_{\Phi_{\epsilon_n}}(\mu_{v\cdot \Phi_{\epsilon_n}})$ have computable global moduli of continuity $\mu_n$ and $\chi_n$, respectively.  Fix an integer $k>0$.  Let $\delta=2^{-\min\{\mu(k+1),\chi(k+1)\}}$.  By Lemma \ref{lem:covering}, we can find $\{p_1,\dots,p_t\}$ so that, for any $q\in B$, there is some $p_i$ so that $\|q-p_i\|<\delta$.  For $q\in B(p_i,\delta)$, by the definition of $\chi$, $\|\rv_{\Phi_{\epsilon_n}}(\mu_{q\cdot\Phi_{\epsilon_n}})-\rv_{\Phi_{\epsilon_n}}(\mu_{p_i\cdot\Phi_{\epsilon_n}})\|<2^{-k-1}$.  Moreover, since $\psi$ an oracle for $w_0$, the map $v\mapsto \rv_{\Phi_{\epsilon_n}}(\mu_{v\cdot \Phi_{\epsilon_n}})$ is computable, and distance function in $\mathbb{R}^m$ is computable, we compute $d_i$, which is an approximation of the distance between $w_0$ and $\rv_{\Phi_{\epsilon_n}}(\mu_{p_i\cdot\Phi_{\epsilon_n}})$, with error at most $2^{-k-1}$.  Therefore, we have the following bounds on the distance: $d_i-2^{-k}<\|\rv_{\Phi_{\epsilon_n}}(\mu_{q\cdot\Phi_{\epsilon_n}})-w_0\|<d_i+2^{-k}$.

We now prove that the maps $(n,s)\mapsto h^l_{\Phi_{\epsilon_n}}(w_0,2^{-s})$ and $(n,s)\mapsto h^u_{\Phi_{\epsilon_n}}(w_0,2^{-s})$ are computable.  This, combined with Theorem \ref{thmHcompgen} proves that $\cH$ is a computable function.  Let $s$ be a real number and $\tau$ an oracle for $s$.  Suppose that we use the oracle $\tau$ to compute an approximation $s_k$ to $s$ of error at most $\log_2(1+2^{-k+s})$, from this inequality, we find that $|2^{-s}-2^{-s_k}|<2^{-k}$.  Let $i_1,\dots,i_\ell$ be the subset of the indices $1,\dots,t$ such that $d_i<2^{-s_k}+2^{-k+1}$; this inequality implies that $d_i-2^{-k}<2^{-s}$.  By construction, for all $i\not\in\{i_1,\dots,i_\ell\}$, $\rv_{\Phi_{\epsilon_n}}(\mu_{q\cdot\Phi_{\epsilon_n}})$ is excluded from $B(w_0,2^{-s})$ for all $q\in B(p_i,\delta)$.  Additionally, for all $j=1,\dots,\ell$, by construction, $\rv_{\Phi_{\epsilon_n}}(\mu_{q\cdot\Phi_{\epsilon_n}})$ is contained within $B(w_0,2^{-s}+2^{-k+2})$ for all $q\in B(p_{i_j},\delta)$.  Moreover, by our construction, if $q\in\mathbb{R}^m$ so that $\rv(\mu_{q\cdot\Phi_{\epsilon_n}})\in B(w_0,2^{-s})$, then $q\in B(p_{i_j},\delta)$ for some $j=1,\dots,\ell$.

Observe that, by the definition of $\delta$, if $q\in B(p_{i_j},\delta)$, then $|h_{\mu_{q\cdot\Phi_{\epsilon_n}}}(f)-h_{\mu_{p_{i_j}\cdot\Phi_{\epsilon_n}}}(f)|<2^{-k-1}$.  Since the map $v\mapsto h_{\mu_{v\cdot \Phi_{\epsilon_n}}}(f)$ is computable, we can compute  $h_{i_j}$, an approximation to $h_{\mu_{p_{i_j}\cdot\Phi_{\epsilon_n}}}(f)$ with error at most $2^{-k-1}$.  Therefore, it follows that $\min_j\{h_{i_j}\}-2^{-k}\leq h_{\mu_{q\cdot\Phi_{\epsilon_n}}}(f)\leq \max_j\{h_{i_j}\}+2^{-k}$.  Combining all of these inequalities, it follows that 
\begin{eqnarray*}
h_{\Phi_{\epsilon_n}}^l(w_0,2^{-s})\geq \min_j\{h_{i_j}\}-2^{-k} \geq h_{\Phi_{\epsilon_n}}^l(w_0,2^{-s}+2^{-k+2})-2^{-k+1}\, {\rm and}\\
h_{\Phi_{\epsilon_n}}^u(w_0,2^{-s})\leq \max_j\{h_{i_j}\}+2^{-k} \leq h_{\Phi_{\epsilon_n}}^u(w_0,2^{-s}+2^{-k+2})+2^{-k+1}.
\end{eqnarray*}
By repeating this argument with $2^{-s}-2^{-k+2}$ substituted for $2^{-s}$ results in a new subset of indices $i'_1,\dots,i'_{\ell'}$ so that the following inequalities hold:
\begin{multline*}
h_{\Phi_{\epsilon_n}}^l(w_0,2^{-s}-2^{-k+2})+2^{-k+1}\geq \min_{j'}\{h_{i'_{j'}}\}+2^{-k}\geq h_{\Phi_{\epsilon_n}}^l(w_0,2^{-s})\\
\geq \min_j\{h_{i_j}\}-2^{-k} \geq h_{\Phi_{\epsilon_n}}^l(w_0,2^{-s}+2^{-k+2})-2^{-k}
\end{multline*}
and
\begin{multline*}
h_{\Phi_{\epsilon_n}}^u(w_0,2^{-s}-2^{-k+2})-2^{-k+1}\leq \max_{j'}\{h_{i'_{j'}}\}-2^{-k}\leq h_{\Phi_{\epsilon_n}}^u(w_0,2^{-s})\\
\leq \max_j\{h_{i_j}\}+2^{-k} \leq h_{\Phi_{\epsilon_n}}^u(w_0,2^{-s}+2^{-k+2})+2^{-k}.
\end{multline*}
The continuity of $\cH$, which implies the continuity of $h_{\Phi_{\epsilon_n}}^u$ and $h_{\Phi_{\epsilon_n}}^l$, implies that as $k$ increases, the outer terms of these inequality converge to the desired inner term.  Therefore, the maximum and minimum terms converge to the desired middle term.  Moreover, since the desired inner term is between the maximum and minimum terms, respectively, when the difference between these terms is sufficiently small, their average is an approximation to $h_{\Phi_{\epsilon_n}}^{u/l}(w_0,2^{-s})$.  This completes the conditions for Theorem \ref{thmHcompgen} and hence $\cH$ is computable.
\end{proof}

\begin{remark} We note that, in Theorem \ref{thmmaingenentcomp}, we can replace the assumption on the uniqueness of the equilibrium states of the potentials $v\cdot \Phi_{\epsilon_n}$ by a slightly more general condition: Namely,
it is sufficient to require that for all $v\in \bR^m$ and all $n\in \bN$, the equilibrium states of the potentials $v\cdot \Phi_{\epsilon_n}$ have the same rotation vector.  When the rotation vectors agree, the equilibrium states also have the same entropy. The more general condition holds if and only if $v\mapsto P_{\rm top}( v\cdot \Phi_{\epsilon_n})$ is differentiable on $\bR^m$ for all $n\in \bN$ (see \cite{Je}).
\end{remark}

\begin{definition}
Let $f:X\to X$ be a computable map on a compact computable metric space.   We say that the {\em periodic points of $f$ are uniformly computable} if there exists a Turing machine $\phi=\phi(n,k)$ such that given inputs $n,k\in \bN$,  $\phi(n,k)$ is a finite set of natural numbers $\{i_1,\dots,i_\ell\}$ such such that the Hausdorff distance $d_H(\Per_n(f),\{s_{i_1},\dots,s_{i_\ell}\})<2^{-k}$.
\end{definition}

\begin{remark} For example, if $f:\overline{\bC}\to \overline{\bC}$ is a polynomial with rational coefficients on the Riemann sphere $\overline{\bC}$, then the periodic 
points of $f$ are uniformly computable since the complex roots of a polynomial can be approximated, see, e.g., \cite{Pinkert,Wilf,Brunetto,SagraloffYap}.
\end{remark}

Recall that $\Per_n(f)$ denotes the set of periodic points of $f$ with prime period $n$, see Section 1.2.
The following result provides a method for the  computation of the radius $r$  in Theorem \ref{thmmaingenentcomp}.
\begin{theorem}\label{thmrcomp}
Let $f:X\to X$ be a computable map on a compact computable metric space such that $\cM_{\rm Per}$ is dense in $\cM_f$. Suppose that the periodic points of $f$ are uniformly computable.  Let $\Phi:X\to \bR^m$ be computable.  Then, there exists a Turing machine that, for a  given  oracle $\psi$ for $w_0\in \inn \R(\Phi)$, produces a positive rational number $r(w_0)$ such that $B(w_0,r(w_0))\subset\inn\R(\Phi)$.
\end{theorem}
\begin{proof}
Since $X$ is compact and both $f$ and $\Phi$ are computable, there are computable global moduli of continuity, $\mu$ and $\chi$, for $f$ and $\Phi$, respectively.  Let $n$ and $k$ be positive integers.  Define $b=\min_{0\leq i<n}\mu^i(\chi(k))$.  Then, observe that if $x\in\Per_n(f)$ and $y\in X$ so that $d_X(x,y)<2^{-b}$, then each of the first $n$ iterates of $x$ and $y$ are within $2^{-\chi(k)}$, so that for $i=0,\dots,n-1$, $\|\Phi(f^i(x))-\Phi(f^i(y))\|<2^{-k}$.

Recall that $\mu_x$ denotes the invariant measure supported on the orbit of the periodic point $x$.  In this case, $\rv_{\Phi}(\mu_x)=\frac{1}{n}\sum_{i=1}^{n-1} \Phi(f^i(x))$.  Since the periodic points of $f$ are uniformly computable, by computing $\phi(n,b)$, we can find an approximation $s_{i_j}$ to $x$ so that $\left\|\rv_\Phi(\mu_x)-\frac{1}{n}\sum_{i=1}^{n-1} \Phi(f^i(s_{i_j}))\right\|<2^{-k}$.  Moreover, since $s_{i_j}$ is computable and $\Phi$ and $f$ are both computable, we can approximate $\Phi(f^i(s_{i_j}))$ by $a_i$ so that $\left\|a_i-\Phi(f^i(s_{i_j}))\right\|<2^{-k}$ for all $i=0,\dots,n-1$.  Therefore, $\left\|\rv_{\Phi}(\mu_x)-\frac{1}{n}\sum a_i\right\|<2^{-k+1}$, so $\rv_{\Phi}(\mu_x)$ can be approximated to any accuracy.

By the density of the periodic point measures in $\cM_f$, we know that there exists some $n>0$ such that the convex hull $C_n=\conv\{\rv(\mu_x): x\in \Per_1(f)\cup\dots\cup \Per_n(f)\}$ contains $w_0$ in its interior.  Using the approximations to $\rv_{\Phi}(\mu_x)$ and any standard convex hull algorithm, we can compute an approximation of $C_n$ up to Hausdorff distance $2^{-n}$, let this approximation be $A_n$.  By successively computing $A_n$ and $\psi(n)$ for increasing $n$, we can find an $n$ so that $w_0$ is guaranteed to be within $\inn C_n$ as follows: Since the boundary of $A_n$ is defined by hyperplanes passing through points with rational coordinates, we can compute $a_n$, which approximates the distance between $w_0$ and $\partial A_n$ to an accuracy of $2^{-n}$.  When $n$ is large enough so that $a_n>2^{-n+1}$, then $w_0$ is guaranteed to be within $C_n$ since $r(w_0)=a_n-2^{-n+1}$ is a lower bound on the distance between $w_0$ and $\partial C_n$.  Then, since $C_n\subset\R(\Phi)$, this $r(w_0)$ has the required property.
\end{proof}

\section{Computability of Rotation Sets for Shift Maps}

\subsection{Shift maps }\label{sec:shiftmaps}
We collect some basic facts on shift maps. Let $d\in \bN$, and let $\cA=\{0,\dots,d-1\}$ be a finite alphabet of $d$ symbols. The (one-sided) {\em shift space} $\Sigma_d$ on the alphabet $\cA$ is the set of
all sequences $x=(x_k)_{k=1}^\infty$ where $x_k\in \cA$ for all $k\in \bN$.  We endow $\Sigma_d$ with the {\em Tychonov product }topology
which makes $\Sigma_d$ a compact metrizable space. For example, given $0<\theta<1$, the metric given by
\begin{equation}\label{defmet}
d(x,y)=d_\theta(x,y)\eqdef\theta^{\min\{k\in \bN:\  x_k\not=y_k\}}\qquad\text{and}\qquad d(x,x)=0
\end{equation}
induces the Tychonov product topology on $\Sigma_d$.
The {\em shift map} $f:\Sigma_d\to \Sigma_d$, defined by $f(x)_k=x_{k+1}$, is a continuous $d$-to-$1$ map on $\Sigma_d$.

If $X\subset \Sigma_d$ is an $f$-invariant set, we  say  that $f|_X$ is a {\em sub-shift} with shift space $X$.  In the following, we use the symbol $X$ for any shift space including the full shift $X=\Sigma_d$.               
A particular class of subshifts are subshifts of finite type.
Namely, suppose $A$ is a $d\times d$ matrix with values in $\{0,1\}$, then consider the set of sequences given by  
$X=X_A=\{x\in \Sigma_d: A_{x_k,x_{k+1}}=1\}$.  $X_A$ is a closed (and, therefore, compact) $f$-invariant set, and we say that $f|_{X_A}$ a {\em subshift of finite type}. By reducing the alphabet, if necessary, we always assume that $\cA$ does not contain letters that do not occur in any of the sequences in $X_A$. 

A continuous map $f:Y\to Y$ on a compact metric space $Y$ is called {\em (topologically) transitive} if, for any pair of non-empty open sets $U,V\subset Y$, there exists $n\in \bN$ such that $f^n(U)\cap V\not=\emptyset$. Note that if $f$ is onto, then transitivity is equivalent to having a dense orbit. 
Moreover, we say $f: Y \rightarrow Y$ is {\em topologically mixing}, if for any pair of non-empty open sets $U,V\subset Y$, there exists $N\in \bN$ such that $f^n(U)\cap V\not=\emptyset$ for all $n\geq N.$
A subshift of finite type $f|_{X_A}$ is transitive if and only if $A$ is {\em irreducible}, that is, for each $i,j$, there exists an $n\in \bN$ such that $A^n_{i,j}>0$.  Moreover,  $f|_{X_A}$ is topologically mixing if and only if $A$ is {\em aperiodic}, that is, there exists $n\in \bN$ such that $A^n_{i,j}>0$ for all $i,j$.

In the context of symbolic dynamics, we always consider the case one-sided shift maps.
However, all our result carry over to the case of two-sided shift maps. For details how to make the
connection between one-sided shift maps and two-sided shift maps we refer to
\cite{Je}.

Let $f:X\to X$ be a one-sided subshift.
Given $x\in X$, we write $\pi_k(x)=(x_1,\dots,x_k)\in \cA^k$. Let $\tau=(\tau_1,\dots,\tau_k)\in \cA^n$.  We denote the {\em cylinder of length $k$ generated} by $\tau$ by $\cC(\tau)=\{x\in X: x_1=\tau_1,\dots, x_k=\tau_k\}$. Note that $\cC(\tau)$ may be empty. If  $\Or(\tau)\eqdef(\tau_1,...,\tau_k,\tau_1,...,\tau_k,...)\in X$
we call $\Or(\tau)$ the {\em periodic point in $X$ generated by $\tau$ of period $k$}.
We say $x\in X$ is a preperiodic point if $f^n(x)$  is periodic for some $n\in \bN$.
If $f$ is a subshift of finite type and  $\cC(\tau)$ is not empty, then the preperiodic points in $\cC(\tau)$ form a  dense (and, in particular, nonempty) subset of $\cC(\tau)$.  

Given $x\in X$ and $k\in \bN$, we call $\cC_k(x)=\cC(\pi_k(x))$ the {\em cylinder of length $k$ generated by $x$}. 
We denote the {\em cardinality of the set of cylinders of length $k$ in $X$} by $m_c(k)$. Note that $m_c(k)\leq d^k$, with equality for the full shift.
 
Recall the definitions in Section 1.2 for  $\Per_n(f)$ and $\Per(f)$, i.e.,  the set of periodic points with prime period $n$  and the set of periodic points of $f$ respectively. Let $x\in \Per_n(f)$.  We call $\tau_x=(x_1,\dots,x_{n})$ the {\em generating segment} of  $x$, i.e., $x=\Or(\tau_x)$.  
Next, we define certain periodic points that play a crucial role when dealing with locally constant potentials.

\begin{definition}\label{defelementary} 
Let $k\in \bN$ and $n\geq 1$. We say $x\in\Per_n(f)$ is a {\em $k$-elementary periodic point with period $n$} if the cylinders $\cC_k(x),\dots,\cC_k(f^{n-1}(x))$ are pairwise disjoint.  Fixed points 
are the $k$-elementary points with period $1$. In the case $k=1$, we simply say $x$ is an {\em elementary periodic point}. We denote by $\EPer^k(f)$ the set of all $k$-elementary periodic points.
\end{definition}
  We observe that the period $n$ of a $k$-elementary periodic point is at most 
$n \leq m_c(k)\leq d^k$. In particular,  $\EPer^k(f)$ is finite. 

Let  $\Phi\in C(X,\bR^m)$.  Given $k\in\bN$, we define the {\em k-variation} of $\Phi$ as the maximum difference of $\Phi$ applied to elements of the same cylinder of length $k$, i.e., $\var_k(\Phi)=\sup\{\|\Phi(x)-\Phi(y)\|: x_1=y_1,\dots, x_k=y_k\}$.
We say $\Phi$ is constant on cylinders of length $k$ if $\var_k(\Phi)=0$. It is easy to see that $\Phi$ is locally constant if and only if $\Phi$ is constant on cylinders of length $k$ for some $k\in\bN$.   

\subsection{Computability of Shift Spaces and Potentials}

Throughout this section, we assume that $X$ is a shift space on the alphabet $\mathcal{A}=\{0,\dots,d-1\}$.  Since, in general, there are uncountably many elements of $X$, we develop a computability theory for shift spaces.  We begin by explicitly applying the definition of computability from Definition \ref{def:computable} in this case:

\begin{definition}
An {\em oracle} for a sequence $x=\{x_n\}_{n \in \mathbb{N}}\in X$ is a function $\phi$ such that on input $n$, $\phi(n)=x_n$.  Moreover, $x$ is called {\em computable} if there is a Turing machine $\phi$ which is an oracle for $x$.
\end{definition}

Since, in general, $X$ is uncountable and there are only countably many Turing machines, most sequences are not computable; preperiodic sequences, however, are computable.  Additionally, the definition of a computable function is identical to Definition \ref{def:computable}.  Observe that the distance function $d_\theta$ generating the Tychonov product topology, see Section \ref{sec:shiftmaps}, is computable if and only if $\theta\in (0,1)$ is computable.  Therefore, throughout the remainder of this paper, we assume that $\theta$ is a computable real number in $(0,1)$.  

We observe that when $X$ is a shift space for a subshift of finite type, a function $\Phi$ which is locally constant is computable if and only if its range is a  set of computable numbers.  In particular, recall that every nonempty cylinder $C(\tau)$ contains a computable point and the value of $\Phi$ on a computable point is computable.

Additionally, the definition of computable sets carries over directly to the case of shifts spaces.  We note that if $ X$ is a shift space for a subshift of  finite type, then $X$ is computable because, for any fixed $k$, we can enumerate all cylinders of length $k$ which occur in sequences in $X$.  Each of these cylinders is of the form $\mathcal{C}(\tau)$ for some $\tau$, and $\mathcal{C}(\tau)$ is an open ball containing a preperiodic point.  In addition, we can apply Lemma \ref{globalmodulusofcontinuity} to computable sets to prove the existence of a computable global modulus of continuity for computable functions.

Recall the definition of the total parameter space $T\subset C(X,\mathbb{R}^m)\times\mathbb{R}^m$ given in Definition \ref{def:parameter}.  We make the definition of an oracle for a point in this space explicit as follows:

\begin{definition}
Let $(\Phi, w)$ be a point in $T$, i.e., $w\in\R(\Phi)$.  An {\em oracle} for $(\Phi, w)$ is a function $\psi$ such that, for any natural number $n$, $\psi(n)$ returns a rational number which is within $2^{-n}$ of $w$ and a function $\phi:X\rightarrow\mathbb{R}^m$ such that for all $x\in X$, $\phi(x)$ and $\Phi(x)$ differ by at most $2^{-n}$.  Moreover, $(\Phi, w)$ is called {\em computable} if both $\Phi$ and $w$ are computable.  In this case, $\psi$ must be a Turing machine and the function $\phi$ is an oracle Turing machine for $\Phi$.
\end{definition}

For a distance on $T$, we use the sum of the supremum norm on $C(X,\mathbb{R}^m)$ and the Euclidean distance on $\mathbb{R}^m$. We also explicitly define computable functions on subsets of $T$.

\begin{definition}
Let $U\subset T$ and $F:U\rightarrow\mathbb{R}$.  Then, $F$ is {\em computable} if there is a Turing machine $\chi$ such that for any $(\Phi,w)\in U$ and any oracle $\psi$ for $\Phi$, oracle $\eta$ for $w$, and natural number $n$, $\chi(\psi,\eta,n)$ is a rational number of distance at most $2^{-n}$ to $F(\Phi, w)$.
\end{definition}

\subsection{Locally constant potentials for subshifts of finite type}
Let $f:X\to X$ be a subshift of finite type on the alphabet $\cA=\{0,\dots,d-1\}$ with transition matrix $A$. For $k, m\in \bN$, we denote the set of potentials $\Phi:X\to\bR^m$ that are constant on cylinders of length $k$ by $LC_k(X,\bR^m)$. 
Based on work of  Ziemian \cite{Z} and Jenkinson \cite{Je}, we provide the necessary tools for the study of the  computability of rotation sets and their entropies. We start with the following key result:

\begin{proposition}\label{propone}
 Let $k,m\in \bN$, $\Phi\in LC_k(X,\bR^m)$, and  $d'=m_c(k)$. Then there exist a subshift $g:Y\to Y$ of finite type with alphabet $\cA'=\{0,\dots,d'-1\}$ and 
transition matrix $A'$, and a homeomorphism $h:X\to Y$ that conjugates\footnote{We exclusively use $h_{\mu}$ and $h_{\topo}$ for entropies and $h$ for the conjugate map.  Both notations are fairly common in the respective literature.} $f$ and $g$ (i.e., $h\circ f=g\circ h$) such that the transition matrix $A'$ has at most $d$ non-zero entries in each row, and the potential $\Phi'=\Phi\circ h^{-1}$  is constant on cylinders of length one.
\end{proposition}
\begin{proof}
First, we define  the map $h$. Let $\{\cC_k(0),\dots,\cC_k({m_c(k)-1})\}$ denote the set of cylinders of length $k$ in $X$, which we identify with $\cA'=\{0,\dots, d'-1\}$. The transition matrix $A'$  is defined by  $a'_{i,j}=1$ if and only if there exists $x\in X$ with $\cC_k(x)=i$ and $\cC_k(f(x))=j$. 

Let $Y=Y_{\cA'}$ be the shift space in ${\cA'}^\bN$ given by the transition matrix $A'$. Furthermore, let $g:Y\to Y$ be the corresponding map for the subshift of finite type. For $x\in X$ we define $h(x)=y=(y_n)_{n=1}^\infty$ by $y_n=\cC_k(f^{n-1}(x))$. It follows from the definition that  $h:X\to Y$ is a bijection. Next, we show that $h$ is a homeomorphism. Let $(\xi^n)_{n=1}^\infty$ be a  sequence in $X$ with $\xi=\lim \xi^n$. This means that for any $\ell\in\bN$ there exists $N\in \bN$ such that for all $n\geq N$, $\xi^n_i=\xi_i$ for all $i=1,\dots,\ell$. We conclude that $h(\xi^n)_i=h(\xi)_i$ for $i=1,\dots,\ell+1-k$ for all $n\geq N$, which establishes the continuity of $h$. Finally, since $h$ is a continuous bijection with compact domain, $h$ is a homeomorphism.   Let $x=(x_n)_{n=1}^\infty\in X$ and $y=(y_n)_{n=1}^\infty=h(f(x))$. By definition, $y_n=\cC_k(f^{n-1}(f(x)))=\cC_k(f^n(x))$. On the other hand, $g(h(x))_n=h(x)_{n+1}=\cC_k(f^{(n+1)-1}(x))=\cC_k(f^n(x))=y_n$.
 This shows that $h\circ f=g\circ h$.
 
The assertion  that $A'$ has at most $d$ non-zero entries in each row follows from the fact that for each cylinder $\cC_k(i)\subset X$ the set $f(\cC_k(i))$ can be written as the disjoint union of at most $d$ cylinders of length $k$.  More precisely, for any $x\in f(\cC_k(i))$, the first $k-1$ letters of $x$ are determined by $\cC_k(i)$ and there are at most $d$ possible letters for the $k^{\text{th}}$ position.  Let $y,\widetilde{y}\in Y$ with $y_1=\widetilde{y}_1$. Then $h^{-1}(y),h^{-1}(\widetilde{y})\in\cC_k(i)$ for some $i\in\{1,\dots,m_c(k)-1\}$. Since $\Phi$ is constant on cylinders of length $k$ we may conclude that $\Phi(h^{-1}(y))=\Phi(h^{-1}(\widetilde{y}))$. This shows that $\Phi'$ is constant on cylinders of length one. 
\end{proof}
Ziemian \cite{Z} proved that the rotation set of a potential $\Phi$ that is constant on cylinders of length two is a polyhedron. This result extends  to potentials that are constant on cylinders of length $k\geq 2$, see Jenkinson \cite{Je}. For completeness we provide a short proof.

\begin{theorem}\label{thmelementary}
Let $f:X\to X$ be a transitive subshift of finite type and let $\Phi\in LC_k(X,\bR^m)$. Then $\R(\Phi)$ is a polyhedron, in particular $\R(\Phi)$ is the convex hull
of $\rv(\{\mu_x: x\in \EPer^k(f)\})$.
\end{theorem}
\begin{proof}
Let $(Y,g,\Phi')$ be as in Proposition \ref{propone} and $h$ the conjugate map. Since being transitive is a topological property that is preserved by topological conjugation, it follows that $g$ is also transitive.  
It follows, from the definition, that $h(\EPer^k(f))=\EPer^1(g)$. Therefore, by Proposition \ref{propone}, it suffices to prove the statement for $g$ and $\Phi'$.
Let $y\in Y$ be a periodic point of $g$. The generating segment of $y$ can be obtained as a finite concatenation of generating segments of 
elementary periodic points. Therefore, $\rv_{\Phi'}(\mu_y)$ is a convex combination of $\rv_{\Phi'}(\mu)$ for $\mu\in\EPer^1(g)$; in other words, $\rv_{\Phi'}(\mu_y)\in \conv(\rv(\EPer^1(g)))$. The result now follows from the facts that the periodic point measures are dense in $\cM$ and the finiteness of $\EPer^1(g)$. 
\end{proof}

\begin{remark}
From the proof of  
Theorem \ref{thmelementary}, one deduces that the conclusions of Theorem \ref{thmelementary} holds
if $f$ is a subshift with the following properties: 
\begin{itemize}
\item[(i)] $\{\mu_x: x\in \Per(f)\}$ is dense in $\cM$;
\item[(ii)] There exists a finite set $P\subset \Per(f)$ such that the rotation vector of every periodic point measure can be written as a convex sum of rotation vectors of periodic point measures in $P$.
\end{itemize}
For the discussion of property (i), we refer the interested reader to \cite{GK} and the references therein.  To the best of our knowledge, property (ii)  has not been studied in the literature
beyond subshifts of finite type and $k$-elementary periodic points. In principle, however, one can check property (ii) for particular classes of shift maps.
\end{remark}

\begin{theorem}\label{thmcomprotshift}
Let $f:X\rightarrow X$ be a transitive subshift of finite type with computable distance $d_{\theta}$.  Let $\Phi\in C(X,\mathbb{R}^m)$ be computable, then $\R(\Phi)$ is computable.
\end{theorem}
\begin{proof}
To prove this theorem, we check the conditions of Theorem \ref{thmrotcomp}.  Since the distance $d_{\theta}$ is computable, it follows that $\theta$, from the Tychonov product topology, is computable.  The partition from Theorem \ref{thmrotcomp} is formed from cylinders of length $k$ so that the diameter of each cylinder is bounded by $2^{-\mu(n+1)-1}$.  We observe that the cylinders of length $k$ can be listed by an algorithm.  Since $\theta$, $\mu$, and the logarithm function are all computable, there is a Turing machine to compute a long enough cylinder length $k$.  In each cylinder $\cC(\tau)$, there is a periodic point in $\mathcal{O}(\tau)$.  Such a point can be computed  in finite time since there are only finitely many symbols.  Periodic points are computable, and this provides the $p_i$'s.  Finally, the second condition follows from the fact, that the periodic point measures are dense in $\cM$ and from Theorem \ref{thmelementary}. In particular, it is shown in Theorem  \ref{thmelementary} that every periodic point measure is a convex combination of the invariant measures supported on the elementary periodic points.  Therefore, by Theorem \ref{thmrotcomp}, $\R(\Phi)$ is computable.
\end{proof}

Next, we prove the following useful lemma:
\begin{lemma}\label{lemappro}
Let $\Phi\in C(X, \bR^m)$. For all $\epsilon>0$, there exists an $\epsilon$-approximation $\Phi_\epsilon$ of $\Phi$
such that $\Phi_\epsilon$ is constant on cylinders of length $k(\epsilon)$ for some $k(\epsilon)\in \bN$. If $\Phi$ is not locally constant, then $k(\epsilon)\to\infty$ as $\epsilon\to 0$.
Moreover, if $\Phi$ and $\epsilon$ are computable then $\Phi_\epsilon$ and $k(\epsilon)$ can be chosen to be computable.
\end{lemma}
\begin{proof}
Since $X$ is compact, $\Phi$ is continuous, and cylinders form a basis for the topology on $X$, for any $\varepsilon$, we can find a $k$ such that in every cylinder $\cC(\tau)$ of length $k$, if $x,y\in C(\tau)$, then $\|\Phi(x)-\Phi(y)\|<\epsilon$.  From this $\Phi_\epsilon$ can be constructed by choosing $x_\tau\in \cC(\tau)$ for each nonempty cylinder of length $k$ in $X$ and setting $\Phi_\epsilon(x)=\Phi(x_\tau)$ whenever $x\in \cC(\tau)$.

Suppose that $\Phi$ is not locally constant, and fix $k>0$.  Since $\Phi$ is not locally constant, for some cylinder $\cC(\tau)$ length $k$, there exist $x,y\in\cC(\tau)$ such that $\Phi(x)\not=\Phi(y)$.  So we may choose $0<\epsilon<\frac{1}{2}\|\Phi(x)-\Phi(y)\|$ and get $k(\epsilon)>k$.

Finally, if $\epsilon$ is computable, we can find, a sufficiently close approximation to $\epsilon$, an $n$ so that $2^{-n-1}<\epsilon$.  Moreover, if $\Phi$ is computable, it has a computable global modulus of continuity, so there is a Turing machine $\mu$ so that if $d_\theta(x,y)<2^{-\mu(n+1)}$, then $\|\Phi(x)-\Phi(y)\|<2^{-n-1}<\epsilon$.  Since the diameter of cylinder of length $k$ is $\theta^{k+1}$, and $\theta$, $\mu$, and the natural logarithm functions are all computable, we can compute a $k$ so that $\theta^{k+1}<2^{-\mu(n+1)}$.  Moreover, once $k$ has been computed, we can use the construction from Theorem \ref{thmrotcomp} to construct oracle Turing machines for periodic points $\mathcal{O}(\tau)$ in each nonempty cylinder.  Then, $\Phi_\epsilon$ can be constructed for each cylinder $\cC(\tau)$ by approximating $\Phi(\mathcal{O}(\tau))$ to within $2^{-n-1}$ since the image of a computable point under a computable function is computable.  Then, we define $\Phi_\epsilon(x)$ to be this approximation whenever $x\in\cC(\tau)$.
\end{proof}

\section{Computability of Localized Entropy for Shift Maps}\label{sec:6}

In this section, we build upon the results of Section \ref{sec:computability:localized} to prove that the localized entropy function is computable on the interior of the rotation set for transitive subshifts of finite type.  Our main goal is to verify the hypotheses of Theorem \ref{thmmaingenentcomp} in this case.  By applying Lemma \ref{lemappro}, it is enough to verify the hypotheses of Theorem \ref{thmmaingenentcomp} for locally constant potentials.  For these potentials we establish the computability of the entropies and rotation vectors of their equilibrium states.  Throughout this section, we assume that $f:X\to X$ is a transitive subshift of finite type over an alphabet with $d$ symbols and transition matrix $A=(A_{ij})$.  We recall that transitivity is equivalent to $A$ being irreducible.

We begin by introducing some notation.  For an $n\times m$ matrix $B=(B_{ij})$, we write $B\geq 0$ or $B>0$ if all the entries of $B$ are nonnegative or positive, respectively.  Moreover, for  two $n\times m$ matrices $B$ and $C$, we write $B\geq C$ or $B>C$ if $B-C\geq 0$ or $B-C>0$, respectively.

Next, we review some basic facts about Markov measures, for details, see \cite{Kit}.  We say that a $d\times d$ matrix $B\geq 0$ is {\em compatible} with $A$ provided $A_{ij}=0$ implies $B_{ij}=0$.  Moreover, $B$ is {\em faithfully compatible} with $A$ provided $A_{ij}>0$ if and only if $B_{ij}>0$.  For example, for any constants $\Phi(i,j)\in\bR$, the matrix $B$ with $B_{ij}=e^{\Phi(i,j)}A_{ij}$ is faithfully compatible with $A$.  In fact any faithfully compatible matrix $B$ can be written as $B_{ij}=e^{\Phi(i,j)}A_{ij}$ for some $\Phi:\{1,\dots,d\}^2\rightarrow\bR$. Such a $B$ is irreducible, since $A$ is irreducible.  A $d\times d$ matrix $B\geq 0$ is {\em row stochastic} if the sum of the entries in each row is $1$.  The set of all row stochastic matrices compatible with $A$ is denoted by $M_{\rm stoch} (A)$.

By the Perron Frobenius theorem, any matrix $B\geq 0$ faithfully compatible with an irreducible matrix $A$ has positive left and right eigenvectors $l>0$ and $r>0$ associated to its real Perron-Frobenius eigenvalue $\lambda>0$.  Let the matrix $P=(P_{ij})$ be given by 
\begin{equation} \label{eqn:faithful}
P_{i j}:=B_{i j} \frac{r_j}{\lambda r_i}.
\end{equation}
We observe that this matrix is row stochastic, has Perron-Frobenius eigenvalue $1$, and, since $B$ is faithfully compatible with $A$, $P$ is faithfully compatible with $A$.  It follows that $p=(r_1 l_1,...,r_dl_d) >0$ is a left eigenvector of the Perron-Frobenius eigenvalue $1$ of $P$.  Moreover, if $r$, $l$ are normalized so that their inner product is $1$, i.e., $(r,l)=1$, then $p$ is a probability vector.
In this situation the pair $(p,P)$ defines a probability measure $\mu=\mu_{(p,P)}$ characterized by its value on cylinders. Namely,
\begin{equation}\label{markovdef}
\mu(\cC(j_0, j_1,\dots,j_r))=p_{j_0} P_{j_0 j_1}\dots P_{j_{r-1} j_r}.
\end{equation}
We say $\mu=\mu_{(p,P)}$ is the {\em ($1$-step) Markov measure} associated with $P$.  A direct computation, see e.g., \cite{Kit}, shows  that $\mu$ is $f$-invariant and that its measure 
theoretic entropy  is given by

\begin{equation}\label{entmarkov}
h_\mu(f)=-\sum_{i,j} p_i P_{i j} \log P_{i j}.
\end{equation}
If $P$ is defined as in Equation \eqref{eqn:faithful} for some matrix $B$ faithfully compatible to an irreducible $A\geq 0$, we can write the entropy of $\mu$ in terms of $B$ as follows:
\begin{align*}
h_{\mu}(f)&=-\sum_{i,j}p_iP_{ij}\log P_{ij} =-\sum_{i,j}l_iB_{ij}\frac{r_j}{\lambda} 
(\log B_{ij} + \log r_j -\log r_i -\log \lambda)\\
&=\log \lambda -\sum_{j}l_j r_j\log r_j +\sum_{i}l_i r_i \log r_i -
\sum_{i,j}l_i\frac{r_j}{\lambda}B_{ij}\log B_{ij}\\
&=\log \lambda - \sum_{i,j}l_i\frac{r_j}{\lambda}B_{ij}\log B_{ij}.\stepcounter{equation}\tag{\theequation}\label{eq:entropy:pf}
\end{align*}
We observe that since $A_{i,j}\in \{0,1\}$ for all $i,j=1,...,d$, if $B=A$, then
$$ h_{\mu}(f) = \log \lambda = h_{\rm top}(f), 
$$

In the case where $B_{ij}=e^{\Phi(i,j)} A_{ij}$, by substituting this equation into Equation \eqref{eq:entropy:pf}, it follows that
\begin{align*}
h_{\mu}(f)&=\log \lambda- \sum^d_{i,j=1}l_ir_i\frac{r_j}{\lambda r_i}B_{ij} \Phi(i,j)\\
&=\log \lambda - \sum^d_{i,j=1}p_i P_{ij} \Phi(i,j)
=\log \lambda-\int_X \Phi d\mu. \stepcounter{equation}\tag{\theequation}\label{pressure:topentropy}
\end{align*}
Note, that we can interpret $\Phi$ in the integral $\int \Phi d\mu$ in Equation \eqref{pressure:topentropy} as a real-valued potential defined by $\Phi( \cC(i,j))=\Phi(i,j)$. In particular, $\Phi\in LC_2(X,\bR)$. 
We obtain the following useful result:
\begin{proposition}\label{propmarkequi}
Let $f:X\to X$ be a transitive subshift of finite type with transition matrix $A$, and let $\Phi\in LC_2(X,\bR)$. Let $B\in \bR^{d\times d}$  defined by $B_{ij}=e^{\Phi(\cC(i,j))} A_{ij}$. Let $\lambda$ denote the Perron-Frobenius eigenvalue of $B$. Let $P$ be defined as in Equation \eqref{eqn:faithful}, and let $\mu=\mu_{(p,P)}$ be the Markov measure associated with $P$ defined in Equation \eqref{markovdef}. Then, $\mu$ is the unique equilibrium measure of the potential $\Phi$, that is, the unique invariant measure satisfying
\begin{equation}\label{pressure:topentropy:2}
P_{\rm top}(\Phi)=h_{\mu}(f)+\int_X \Phi d\mu=\log \lambda.
\end{equation}
\end{proposition}
\begin{proof}
The result follows from the Ruelle-Perron-Frobenius Theorem, see, e.g., \cite{Bo, Ru},
the variational principle, i.e., Equation \eqref{eq111}, and Equation \eqref{pressure:topentropy}.
\end{proof}

Next, we consider the computability of the Perron-Frobenius eigenvalue and eigenvectors. We make use of the following elementary facts:

\begin{observation}\label{rem:orthant}
Let $S$ be the closure of the intersection of the $(d-1)$-dimensional sphere $S^{d-1}$ with the first orthant in $\mathbb{R}^d$.  We observe that $S$ is a computable set, see \cite{BY} for a hint.  Moreover, if $\{\overline{B}(s_{k_i},2^{-l_i})\}$ is a collection of closed balls with the property that the Hausdorff distance between the union of balls and $S$ is less than $2^{-n}$, then $2^{-l_i}<2^{-n}$ for all $i$.  Therefore, every $x\in S$ is is within $2^{-n}$ of some $s_{k_i}$.
\end{observation}

\begin{proposition}\label{lem:PF:helper}
Let $M_d$ be the set of nonnegative and irreducible $d\times d$ matrices.  The maps assigning to $B\in M_d$ either the Perron-Frobenius eigenvalue, or 
the left and right eigenvectors with first entry $1$ are computable.
\end{proposition}
\begin{proof}
We follow the approach in \cite{Sc,Gantmacher,Wielandt}.  Let $S$ be the closure of the intersection of the $(d-1)$-dimensional sphere $S^{d-1}$ with the first orthant in $\mathbb{R}^d$, see Observation \ref{rem:orthant}.  Moreover, let $T_k$ be a closed tubular neighborhood of radius $2^{-k}$ around $S$.  Since $B$ is irreducible, $(I+B)^{d-1}>0$.  Moreover, by \cite{Gantmacher,Wielandt}, the \PF\ eigenvalue is
$$
\max_{x\in S}\min_i\frac{(B(I+B)^{d-1}x)_i}{((I+B)^{d-1}x)_i}.
$$
We observe that the numerator $(B(I+B)^{d-1}x)_i$ and denominator $((I+B)^{d-1}x)_i$ are both computable because they only involve matrix multiplication.  Moreover, since $(I+B)^{d-1}>0$, with sufficient precision, the denominator is bounded away from zero.  Therefore, for each $i$, the following map is a computable function:
$$
(B,x,i)\mapsto\frac{(B(I+B)^{d-1}x)_i}{((I+B)^{d-1}x)_i}.
$$
Since there are only finitely many values for $i$, the minimum function is also computable, i.e., 
\begin{equation}\label{eq:PF:eigenvalue}
(B,x)\mapsto\min_i\frac{(B(I+B)^{d-1}x)_i}{((I+B)^{d-1}x)_i}
\end{equation}
is a computable function.

Since $(I+B)^{d-1}>0$, we can compute a lower bound on the entries of $(I+B)^{d-1}$ and use this to compute a $k$ so that for all $x\in T_k$, $(I+B)^{d-1}x>0$.  From this point forward, we use the domain of $M_d\times T_k$ for the function in \eqref{eq:PF:eigenvalue}.  Since $M_d$ is not compact, we cannot use Lemma \ref{globalmodulusofcontinuity} to imply the existence of a computable modulus of continuity.  We now fix $B$ (or an oracle for $B$), and let $f_B$ be the corresponding restriction of the function in \eqref{eq:PF:eigenvalue}.  In this case, $f_B$ is defined over a compact and computable set $T_k$ and has a computable global modulus of continuity $\psi_B$.

Fix an integer $n>k$.  By Observation \ref{rem:orthant}, we can compute a collection of closed balls $\{\overline{B}(s_{k_i},2^{-l_i})\}$ where the Hausdorff distance between the union and $S$ is less than $2^{-\psi_B(n)}$.  Therefore, for any $x\in S$, there is some $i$ so that $\|f_B(x)-f_B(s_{k_i})\|\leq 2^{-n}$.  By approximating $f_B(s_{k_i})$ for all $s_{k_i}$ sufficiently well, we derive upper and lower bounds on the maximum value of $f_B|_S$.  By increasing the value of $n$ until the upper and lower bounds are within $2^{-\ell}$ of each other, we can approximate the \PF\ eigenvalue.

The computation of the left and right eigenvectors is nearly identical, so we only show that one of them is computable.  By \PF\ theory, the \PF\ eigenvalue is an eigenvalue of (algebraic) multiplicity one and all of its entries are nonzero.  Therefore, there is a unique eigenvector $r$ associated to the \PF\ eigenvalue $\lambda$ whose first entry is $1$.  Therefore, after substituting $1$ for the first entry of $r$, $(B-\lambda I)r=0$ is a square system with $d-1$ variables with a unique solution.  This can be solved with Cramer's rule.  Since both $B$ and $\lambda$ can be approximated arbitrarily well, the numerator and denominator in Cramer's rule can be approximated to arbitrary precision.  Moreover, since the system has a unique solution, with enough precision, the denominator can be bounded away from zero.  Therefore, all of the entries in the right eigenvector can be approximated to arbitrary precision.
\end{proof}

Next, we  establish  for  transitive subshifts of finite type the main assumptions of Theorem \ref{thmmaingenentcomp}.

\begin{theorem}\label{computequistate}
Let $f:X\to X$ be a transitive subshift of finite type with transition matrix $A$ and computable distance $d_{\theta}$. Let $\Phi:X\to \bR^m$ be a computable potential.  Suppose that $(\epsilon_n)_n$ is a computable, convergent sequence of positive numbers converging to $0$.  Then, there exists an approximating sequence $(\Phi_{\epsilon_n})_n$ of $\Phi$ such that for all $v\in \bR^m$ and $n\in \bN$ the potential $v\cdot\Phi_{\epsilon_n}$ has a unique equilibrium state $\mu_{v\cdot\Phi_{\epsilon_n}}$. Moreover, the maps $(v,n)\mapsto h_{\mu_{v\cdot\Phi_{\epsilon_n}}}(f)$ and $(v,n)\mapsto\rv(\mu_{v\cdot\Phi_{\epsilon_n}})$ are computable.
\end{theorem} 
\begin{proof}
Since $\Phi$ is computable and $X$ is compact, $\Phi$ has a computable global modulus of continuity $\chi$.  For a fixed $n$, since $\epsilon_n$ is computable, with a sufficiently high approximation, we can find a $k$ so that $2^{-k}<\epsilon_n$.  Since the distance $d_\theta$ is computable, we can find a cylinder length $\ell_n$ so that the diameter of a cylinder of length $\ell_n$ is less than $2^{-\chi(k+1)}$.  For each nonempty cylinder $\cC(\tau)$ of length $\ell_n$, we can find a periodic point $x_\tau$ in this cylinder.  Since periodic points are computable and $\Phi$ is computable, we can find $a_\tau\in \bQ$, an approximation to $\Phi(x_\tau)$, to an accuracy of $2^{-k-1}$.  Then, for $x\in \cC(\tau)$, $\|\Phi(x)-a_\tau\|\leq\|\Phi(x)-\Phi(x_\tau)\|+\|\Phi(x_\tau)-a_\tau\|<2^{-k}<\epsilon_n$.  We define $\Phi_{\epsilon_n}(x)=a_\tau$ for all $x\in \cC(\tau)$.  With this definition, $\|\Phi_{\epsilon_n}-\Phi\|_\infty<\epsilon_n$.  We conclude that $(\Phi_{\epsilon_n})_n$ is an approximating sequence for $\Phi$. Moreover, each $\Phi_{\epsilon_n}$ is a locally constant computable potential with range in $\bQ$.
Since locally constant potentials are, in particular, H\"older continuous, we conclude that for all $v\in \bR^m$ and $n\in \bN$ the potential $v\cdot\Phi_{\epsilon_n}$ has a unique equilibrium state $\mu_{v\cdot\Phi_{\epsilon_n}}$.

Let $\psi$ be an oracle for $v$.  Since $\Phi_{\epsilon_n}$ and $v$ are both computable, the product $v\cdot\Phi_{\epsilon_n}$ is computable.  Moreover, since $\Phi_{\epsilon_n}$ is constant on cylinders of length $\ell_n$, we can approximate $v\cdot\Phi_{\epsilon_n}$ with a potential which is constant on cylinders of length $\ell_n$, let $\Phi_v$ be this approximation.

By applying Proposition \ref{propone} (and constructing a larger alphabet of size $d'\leq \ell_n^d$), we consider the conjugate subshift $g$ of finite type, with transition matrix $A'$ and potential $\Phi'_v$. Recall that $\Phi'_v$ is constant on cylinders of length one.  Moreover, pressure and equilibrium states (after taking the push forward) for $\Phi_v$ and $\Phi'_v$ are preserved under this conjugation, see, e.g., \cite{Wal:81}.  Observe that since $\Phi'_v$ is constant on cylinders of length $1$, it is in particular constant on cylinders of length $2$.

For each cylinder $\cC(i,j)$ of length $2$, we define $B'_{ij}=e^{\Phi'_v(\cC(i,j))}A'_{ij}$, and let $B'=(B'_{ij})$.  Since $g$ is transitive, $A'$ is irreducible, which implies that $B'$ is also irreducible.  Since the exponential function is computable and $\Phi'_v$ can be approximated, based on the oracle $\psi$, $B'$ can be approximated to arbitrary precision.  Moreover, by Lemma \ref{lem:PF:helper}, we can construct the \PF\ eigenvalue $\lambda$ and eigenvectors $r'$ and $l'$ with first coordinate $1$ for $B'$ at arbitrary precision.  Since all entries of $r'$ are computable and nonzero, they can be bounded away from zero, so we can construct the matrix $P'$ where $P'_{ij}=B'_{ij}\frac{r_j}{\lambda r_i}$ to arbitrary precision.  Moreover, we can also construct the probability vector $p'=(r_1'l_1',\dots, r'_{d'}l'_{d'})$ to arbitrary precision.  Let $h$ be the conjugating map between $f$ and $g$. Let $\mu'_v=\mu_{(p',P')}$ be defined as in Equation \eqref{markovdef}. It follows from Proposition \ref{propmarkequi} and Corollary \ref{corcomputrot} that $\mu'_v=h_\ast\mu_v$. Furthermore, applying again Corollary \ref{corcomputrot}, we can calculate the rotation vector of $\mu_v$ with respect to $\Phi_{\epsilon_n}$ by 
\begin{equation} 
\rv(\mu_v)=\left(\int (\Phi'_v)_k d\mu'_v\right)_k=\left(\sum_{i,j} (\Phi'_v(\cC(i,j)))_k p_i P_{ij} \right)_k,
\end{equation}
which can also be approximated to arbitrary precision.  Similarly, the measure theoretic entropy can be calculated to arbitrary precision using Formula \eqref{entmarkov} for $h_{\mu'_v}(f)=h_{\mu_v}(f)$.
\end{proof}

Finally, we are able to prove the main result of this Section:

\begin{theorem}\label{Hintcomp}
Let $f:X\to X$ be a transitive subshift of finite type with computable distance $d_\theta$. Let $\Phi\in C(X,\bR^m)$ be computable.  Then $\cH$ is  computable on $\inn \R(\Phi) $.
\end{theorem}
\begin{proof}
If $\inn \R(\Phi)=\emptyset$ there is nothing to prove. Otherwise, by
applying Theorem \ref{thmrcomp} to $f$ and $\Phi$, we conclude that there exists a computable function $r:\inn \R(\Phi)\to \bR^+$ such that
for all $w\in \inn \R(\Phi)$ we have $B(w,r(w))\subset \R(\Phi)$. Furthermore, Theorem  \ref{computequistate}
establishes the remaining assumptions of Theorem \ref{thmmaingenentcomp}. We now apply Theorem \ref{thmmaingenentcomp} and the result follows.
 \end{proof}

\section{Entropy at boundary points.} \label{sec:7}
In this section, we construct a class of examples that indicate that, in general, localized entropy is not computable at the boundary of rotation sets. More precisely, we show that the global entropy function is, in general, discontinuous at the boundary of the total parameter space of rotation sets, cf Theorem \ref{thmglobentcont}.

 Recall that at an interior point $w_0$  of the rotation set we have
\begin{equation}\label{limul}
\cH_\Phi(w_0)=\lim_{n\to \infty}h^l_{\Phi_{\epsilon_n}}(w_0, \epsilon_n) =\lim_{n\to \infty}h^u_{\Phi_{\epsilon_n}}(w_0, \epsilon_n).
\end{equation}
We  now construct a family of examples for which there is an exposed boundary point where the two limits in Equation \eqref{limul} do not coincide. 

\begin{example}\label{ex1}
Let $f:X\to X$ be the one-sided full shift with alphabet $\{0,1,2,3\}$.  We construct a potential function $\Phi$ as follows: Fix a real number $a>0$ and consider a function $\ell_1:[0,a]\to\bR$ which is continuous, non-negative, increasing, and strictly concave with $\ell_1(0)=0$.  Let $\ell_2=-\ell_1$; therefore, $\ell_2$ is continuous, non-positive, decreasing, and strictly convex with $\ell_2(0)=0$.  Let $(x_k)_{k\in\bN}$ be a sequence with $x_k\in (0,a)$ for all $k$ that is exponentially and strictly decreasing to $0$, see Figure \ref{DefPhi}.  Let $x=\mathcal{O}(1)\in X$ and $y=\mathcal{O}(3)\in X$ be the fixed points of repeating $1$'s and $3$'s, respectively.

\begin{figure}[htb]
\begin{center}
\begin{tikzpicture}[scale=.6]
% Draw the rotation set
\filldraw[fill=lightgray,fill opacity=.25] (0,0) 
	\foreach \x in {.1,.2,.4,.7,1,1.5,2,3,5,7} {-- ({\x},{1.5*sqrt(\x)-1.5*(\x/7)*(\x/7)*(\x/7)*sqrt(\x)})}
	\foreach \x in {5,3,2,1.5,1,.7,.4,.2,.1}{-- ({\x},{-1.5*sqrt(\x)+1.5*(\x/7)*(\x/7)*(\x/7)*sqrt(\x)})}
	--cycle;
%Draw the axes
\draw[->] (-0.5,0) -- (8,0) node[below right]{$x$};
\draw[->] (0,-5.5) -- (0,5.5) node[above right]{$y$};
%Draw the curves l_1 and l_2
\draw[domain=0:2.645751311064591,smooth,variable=\x] plot ({\x*\x},{2*\x}) node[right]{$\ell_1$};
\draw[domain=0:-2.645751311064591,smooth,variable=\x] plot ({\x*\x},{2*\x}) node[right]{$\ell_2$};
%Draw the tick marks
\draw \foreach \x in {7,4.5,2.5,1.5,.75,.5,.3,.2,.1}
{({\x},.15) -- ({\x},-.15)};
%Label the first few tick marks
\node[below] at (7,-.15) {$a$};
\node[below] at (4.5,-.15) {$x_1$};
\node[below] at (2.5,-.15) {$x_2$};
\node[below] at (1.5,-.15) {$x_3$};
\node[below] at (.75,-.15) {$x_4$};
% Draw the marks on l_1 and l_2
\draw \foreach \x in {4.5,2.5,1.5,.75,.5,.3,.2,.1}
{
({\x-.15},{2*sqrt(\x)})--(({\x+.15},{2*sqrt(\x)})
({\x},{2*sqrt(\x)-.15})--(({\x},{2*sqrt(\x)+.15})
({\x-.15},{-2*sqrt(\x)})--(({\x+.15},{-2*sqrt(\x)})
({\x},{-2*sqrt(\x)-.15})--(({\x},{-2*sqrt(\x)+.15})
};
% Draw circles on Rotation set
\draw \foreach \x in {0.002,0.005,0.0125,0.025,.05,.1,.2,.4,.7,1,1.5,2,3,5} 
{
({\x},{1.5*sqrt(\x)-1.5*(\x/7)*(\x/7)*(\x/7)*sqrt(\x)}) circle ({.15*sqrt(sqrt(\x/7)))})
({\x},{-1.5*sqrt(\x)+1.5*(\x/7)*(\x/7)*(\x/7)*sqrt(\x)}) circle ({.15*sqrt(sqrt(\x/7)))})
};
% Label the Rotation set
\node at (3.3,1.3) {$\operatorname{Rot}(\Phi)$};
\end{tikzpicture}
\end{center}
\caption{\label{DefPhi}The rotation set defined in Example \ref{ex1}.  The rotation set is an infinite polygon and the origin is an exposed point.}
\end{figure}
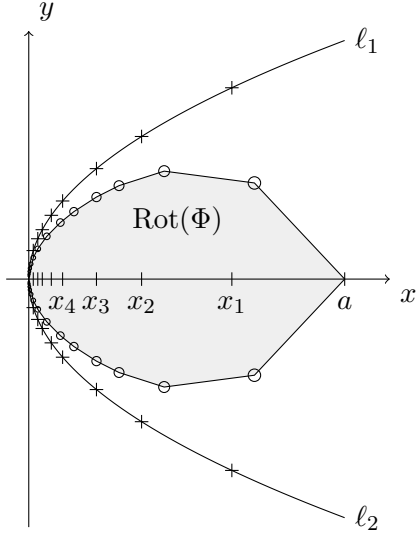
Next, we define several subsets of $X$.  Let $S_1=\{0,1\}$ and $S_2=\{2,3\}$.  Fix a natural number $\lambda\in\bN$ with $\lambda\geq 3$.  For $k\geq \lambda$, define
\begin{align*}
\widetilde{Y}_i(k)&=\{\xi \in X: \xi_1,\dots,\xi_k\in S_i\}\\
Y_i(\infty)&=\{\xi \in X: \xi_j\in S_i \ \mbox{for all} \ j\}=\bigcap_k \widetilde{Y}_i(k)\\
Y_i(k)&=\{ \xi\in \widetilde{Y}_{i}(k-1): \xi\not\in \widetilde{Y}_i(k)\}=\widetilde{Y}_i(k-1)\setminus\widetilde{Y}_i(k).
\end{align*}
For $k>\lambda$, define
\begin{align*}
X_1(k)&=Y_{1}(k)\setminus \cC_{k-1}(x)\\
X_2(k)&=Y_{2}(k)\setminus \cC_{k-1}(y)\\
X(k)&=X_1(k)\cup X_2(k)\\
Y_0(\lambda)&= X\setminus (\widetilde{Y}_1(\lambda)\cup  \widetilde{Y}_{2}(\lambda)).
\end{align*}
We define a potential $\Phi:X\rightarrow\mathbb{R}^2$ by
\begin{equation}\label{defpotphi}
\Phi(\xi)=\begin{cases}
(a,0)& {\rm if}\,\,   \xi\in Y_0(\lambda)\\
                      (x_{k-\lambda},0)    & {\rm if}\,\,  \xi\in X(k), \, k>\lambda\\
                       (x_{k-\lambda},\ell_1(x_{k-\lambda})) & {\rm if} \, \, \xi\in\cC_{k-1}(x) \cap Y_1(k),\, k>\lambda\\
                         (x_{k-\lambda},\ell_2(x_{k-\lambda})) & {\rm if} \, \, \xi\in\cC_{k-1}(y) \cap Y_2(k),\, k>\lambda\\
                       (0,0) & {\rm if}\,\,  \xi\in Y_1(\infty)\cup Y_2(\infty)
            \end{cases}
\end{equation}
\end{example}

Throughout the rest of this section, we study the potential $\Phi$ defined in the example above.

\begin{lemma}\label{phi_kw}
The potential $\Phi$ defined in Equation \eqref{defpotphi}
 is continuous and $(0,0)$ is an exposed point of $\R(\Phi)$. 
\end{lemma}
\begin{proof}
Let $(\xi^n)$ be a convergent sequence in $X$ with $\xi=\lim \xi^n$.  If $\Phi(\xi)=(a,0)$, $(x_{k-\lambda},0)$, or $(x_{k-\lambda},\ell_i(x_{k-\lambda}))$, then the behavior of $\xi$ is determined by the first $k$ terms of $\xi$.  Since $\xi^n$ converges to $\xi$, for $n$ sufficiently large, $\xi^n$ agrees with $\xi$ on the first $k$ terms.  Therefore, $\Phi(\xi^n)=\Phi(\xi)$.

If $\Phi(\xi)=(0,0)$, then $\xi\in Y_i(\infty)$ for $i=1,2$.  For any $k$, for all $n$ sufficiently large, $\xi^n$ and $\xi$ share the first $k$ terms, so $\xi^n\in \widetilde{Y}_i(k)$.  For $k>\lambda$, $\Phi(\xi^n)$ is therefore one of $(x_{k-\lambda},0)$, $(x_{k-\lambda},\ell_i(x_{k-\lambda}))$ or $(0,0)$.  As $k$ grows, the value for all of these expressions approaches $(0,0)$.

To show that $(0,0)$ is an exposed point of $\R(\Phi)$, observe that $\R(\Phi)\subset \conv(\Phi(X))$.  Since all points in the image of $\Phi$ other than $(0,0)$ have positive $x$-coordinate, $(0,0)$ is extremal.  Finally, the compactness of $\Phi(X)$ implies that $(0,0)$ is an exposed point with the $y$-axis as supporting line.
\end{proof}

We now restrict $x_k$ and $\ell_1$ to control the shape of $\R(\Phi)$.  Let $w(0)=(a,0)$ and $w(\infty)=(0,0)$.  Suppose $\sum_{k=1}^\infty x_k<a$ and $\ell_1(x_1)<(\lambda+1)\ell_1(x_2)$.  For $i=1,2$, define
$$
w_i(\lambda)=\frac{1}{3\lambda}[3(\lambda-1)(a,0)+2(x_1,\ell_i(x_1))+(x_1,\ell_{3-i}(x_1))].
$$
Moreover, for $j>\lambda$ and $i\in\{1,2\}$, we define
$$
w_i(j)=\frac{1}{j}\left[\lambda(a,0)+\sum\limits_{k=1}^{j-\lambda}(x_k,\ell_i(x_k))\right]
$$

Using techniques similar to those in the proofs for Example 2 in \cite{KW4}, one can show that
\begin{equation}
\R(\Phi)={\rm Conv}\{w(0),w(\infty), w_i(j): j\ge\lambda,\,i=1,2\}.
\end{equation}
In particular, $\partial \R(\Phi)$ is an infinite polygon.
Furthermore, by requiring additional properties on $\ell_1$, it can be arranged that $w(\infty)$ a smooth boundary point.  We refer the  reader to \cite{KW4} for details.

Next, we define approximations of the potential $\Phi$.  Fix $n\in\bN$ and let $\epsilon_n=2^{-n}$.  Since $\ell_1$ is continuous and $\ell_1(0)=0$, there exists $K=K(n)\geq 2\lambda$ such that $\|(x_{k-\lambda},\ell_1(x_{k-\lambda}))\|<\epsilon_n$ and $\|(x_{k-\lambda},\ell_1(x_{k-\lambda}))-(x_{l-\lambda},\ell_1(x_{l-\lambda}))\|<\epsilon_n$ for all $k,l\geq K$.  
Note that since $\ell_2=-\ell_1$ the corresponding inequalities also hold for points on the curve $\ell_2$. 
We define 
$$
X_{\epsilon_n}=\bigcup_{\lambda<k\leq K}\left(X(k)\cup\left(\cC_{k-1}(x)\cap Y_1(k)\right)\cup(\cC_{k-1}(y)\cap Y_2(k))\right)
$$

Finally, we define the  potentials $\Phi_{\epsilon_n}: X\to \bR^2$ by changing the behavior of the potential function near $Y_1(\infty)\cup Y_2(\infty)$.
\begin{equation}\label{defpotphin}
\Phi_{\epsilon_n}(\xi)=\begin{cases}
\Phi(\xi)\qquad & {\rm if}\,\,   \xi\in X_{\epsilon_n}\cup Y_0(\lambda)\\
                      (x_{K+1-\lambda},0)     & {\rm if}\,\, \xi\in (\widetilde{Y}_1(K)\cup\widetilde{Y}_2(K))\setminus(\cC_{K}(x)\cup\cC_{K}(y))\\
                      (x_{K+1-\lambda},\ell_1(x_{K+1-\lambda})) & {\rm if} \, \, \xi\in \cC_{K}(x) \\
                       (x_{K+1-\lambda},\ell_2(x_{K+1-\lambda})) & {\rm if} \, \, \xi\in \cC_{K}(y) 
                       \end{cases}
\end{equation}

It follows from the construction that $\Phi_{\epsilon_n}$ is constant on cylinders of length $K$ and that $(\Phi_{\epsilon_n})_n$ is a converging sequence of $\epsilon_n$-approximations of $\Phi$.

\begin{theorem}\label{thmfin}\label{thm:diff}
Let $\Phi$ and $(\Phi_{\epsilon_n})_n$ as in Equations \eqref{defpotphi} and \eqref{defpotphin}, then 
\begin{equation}\label{eqdiff}
0=\lim_{n\to \infty}h^l_{\Phi_{\epsilon_n}}(w(\infty), \epsilon_n) <\lim_{n\to \infty}h^u_{\Phi_{\epsilon_n}}(w(\infty), \epsilon_n)=\log 2.
\end{equation}
\end{theorem}
\begin{proof}
First, we prove the right-hand-side identity in Inequality \eqref{eqdiff}.  In Lemma \ref{phi_kw}, we proved that $w(\infty)=(0,0)$ is an exposed point of $\R(\Phi)$.  Moreover, $w(\infty)$ is an extreme point of $\conv(\Phi(X))$, which implies  that each invariant measure $\mu$ with $\rv_\Phi(\mu)=w(\infty)$ must be supported on $Y_1(\infty)\cup Y_2(\infty)$.  Since each $Y_i(\infty)$ is a full shift on two symbols, it follows that $h_{\topo}\left(f |_{Y_1(\infty)\cup Y_2(\infty)}\right)=\log 2$.  From the variational principle, it follows that $\cH_\Phi(w(\infty))=\log 2$.  Therefore, $\lim_{n\to \infty}h^u_{\Phi_{\epsilon_n}}(w(\infty), \epsilon_n)=\log 2$ follows from Proposition \ref{prop101}.

Next we show the left hand-side identity in Inequality \eqref{eqdiff}.  By construction, $(x_{K(n)+1-\lambda},\ell_i(x_{K(n)+1-\lambda}))\in \overline{B}(w(\infty),\epsilon_n)$ for all $n\in \bN$ and $i=1,2$.  We claim $\cH_{\Phi_{\epsilon_n}}(x_{K+1-\lambda},\ell_i(x_{K+1-\lambda}))=0$ as follows: First, we observe that $(x_{K(n)+1-\lambda},\ell_i(x_{K(n)+1-\lambda}))$ is an extreme point of $\R(\Phi_{\epsilon_n})$ since the point is in $\R(\Phi_{\epsilon_n})$ and it is an extreme point of $\conv(\Phi_{\epsilon_n}(X))$.  More precisely, for $i=1$ (the case $i=2$ is analogous), the point is in the rotation set since $\rv_{\Phi_{\epsilon_n}}(\mu_x)=(x_{K(n)+1-\lambda},\ell_1(x_{K(n)+1-\lambda}))$, and it is an extreme point of the convex hull since all other images of $\Phi$ have either a larger $x$-coordinate or the same $x$-coordiante and a smaller $y$-coordinate.

Now, let $\mu\in\cM$ with $\rv_{\Phi_{\epsilon_n}}(\mu)=(x_{K(n)+1-\lambda},\ell_1(x_{K(n)+1-\lambda}))$.  Therefore, the support of $\mu$ is a subset of $\Phi_{\epsilon_n}^{-1}(x_{K(n)+1-\lambda},\ell_1(x_{K(n)+1-\lambda}))$.  Our goal is to show that $\mu=\mu_x$.  Observe that we can write
$$
 \Phi_{\epsilon_n}^{-1}(x_{K(n)+1-\lambda},\ell_1(x_{K(n)+1-\lambda}))= \cC_{K}(x)=\{x\}\cup \bigcup_{k>K} A_k,
$$
where $A_k= \cC_{k-1}(x)\setminus \cC_{k}(x)$.  By construction $f^{-1}(A_k)$ is the disjoint union of $A_{k+1}$ and cylinders of the form $\cC(\xi_1\pi_{k-1}(x)\xi_{k+1})$, where $\xi_1,\xi_{k+1}\in \{0,2,3\}$.  Let $\xi\in \cC(\xi_1\pi_{k-1}(x)\xi_{k+1})$, then $\Phi_{\epsilon_n}(\xi)\not=(x_{K(n)+1-\lambda},\ell_1(x_{K(n)+1-\lambda}))$.  We conclude that $\mu(\cC(\xi_1\pi_{k-1}(x)\xi_{k+1}))=0$, since otherwise $\rv_{\Phi_{\epsilon_n}}(\mu)$ would not be an extreme point.  By the $f$-invariance of $\mu$, it must be that $\mu(A_k)=\mu(A_{k+1})$.  Therefore, $\mu(A_k)=\mu(A_l)$ for all $k,l>K$.  Since the $A_k$'s are pairwise disjoint, $\mu(A_k)=0$ for all $k>K$.  Hence, the support of $\mu$ is $x$ and $\mu=\mu_x$.  It follows that $h^l_{\Phi_{\epsilon_n}}(w(\infty), \epsilon_n)=0$ for all $n\in \bN$, and we obtain $\lim_{n\to \infty}h^l_{\Phi_{\epsilon_n}}(w(\infty), \epsilon_n)=0$.
\end{proof}

\begin{remark}
By choosing $a$, $x_k$ to be computable real numbers, and $\ell_1$ to be a computable function, the result of Theorem \ref{thm:diff} also applies in the computable case.
\end{remark}

\bibliographystyle{plain}

\end{document}